\documentclass{article}
\usepackage{amsmath,amsthm,amsfonts,graphicx}
\usepackage{multirow}
\usepackage{slashbox}
\usepackage{multirow}
\def\smallddots{\mathinner{\raise7pt\hbox{.}\raise4pt\hbox{.}\raise1pt\hbox{.}}}
\def\smallsdots{\mathinner{\raise1pt\hbox{.}\raise4pt\hbox{.}\raise7pt\hbox{.}}}

\DeclareMathOperator{\diag}{diag}

\DeclareMathOperator{\rank}{rank}
\DeclareMathOperator{\nrank}{nrank}

\DeclareMathOperator{\nul}{nul}
\DeclareMathOperator{\nnul}{nnul}

\DeclareMathOperator{\nmb}{nmb}

\numberwithin{equation}{section}
\numberwithin{table}{section}
\newtheorem{theorem}{Theorem}[section]
\newtheorem{lemma}{Lemma}[section]

\newtheorem{corollary}{Corollary}[section]
\newtheorem{fact}{Fact}[section]

\newtheorem{algorithm}{Algorithm}[section]

\newtheorem{definition}{Definition}[section]

\newtheorem{remark}{Remark}[section]

\begin{document}

\title{\bf New Studies of Randomized Augmentation and 
Additive Preprocessing
%Numerical Algorithms for
 %Ill-Conditioned Linear Systems of Equations
\thanks {Some results of this paper have been presented at the 
 ACM-SIGSAM International 
Symposium on Symbolic and Algebraic Computation (ISSAC '2011), San Jose, CA, 2011,
the
3nd International Conference on Matrix Methods in Mathematics and 
Applications (MMMA 2011) in
Moscow, Russia, June 22-25, 2011, 
the 7th International Congress on Industrial and Applied Mathematics 
(ICIAM 2011), in Vancouver, British Columbia, Canada, July 18-22, 2011,
the SIAM International Conference on Linear Algebra,
in Valencia, Spain, June 18-22, 2012, and 
the Conference on Structured Linear and Multilinear Algebra Problems (SLA2012),
in  Leuven, Belgium, September 10-14, 2012.}}

\author{Victor Y. Pan$^{[1, 2],[a]}$
%, Guoliang Qian$^{[2],[b]}$, 
and 
Liang Zhao$^{[2],[b]}$
%and Ai-Long Zheng$^{[2],[d]}$
\and\\
$^{[1]}$ Department of Mathematics and Computer Science \\
Lehman College of the City University of New York \\
Bronx, NY 10468 USA \\
$^{[2]}$ Ph.D. Programs in Mathematics  and Computer Science \\
The Graduate Center of the City University of New York \\
New York, NY 10036 USA \\
$^{[a]}$ victor.pan@lehman.cuny.edu \\
http://comet.lehman.cuny.edu/vpan/  \\
$^{[b]}$ 
lzhao1@gc.cuny.edu \\
} 
 \date{}

\maketitle

% - - - - - - - - - - - - - - - - - - - - - - - - - - - - - - - - - - - - -

\begin{abstract}
\begin{itemize}  
  \item%1
A standard Gaussian random matrix 
has full rank with probability 1 and  
 is well-conditioned
with a probability quite close to 1 and
converging to 1 fast as the matrix 
deviates from square shape and becomes more rectangular. 
 \item%2
If we append sufficiently many standard Gaussian random rows or columns
 to any matrix $A$, such that $||A||=1$, then  
 the augmented matrix 
has full rank with probability 1
 and is well-conditioned  with a probability close to 1,
even if the matrix $A$ is rank deficient or ill-conditioned.
 \item%3
We specify and prove these properties of augmentation and 
extend them 
 to additive preprocessing, that is, to adding 
a product of two rectangular Gaussian matrices.
 \item%4
By applying our randomization techniques
to a matrix that has
numerical rank $\rho$,
we accelerate the known algorithms for 
the approximation of its leading and trailing 
singular spaces associated with its $\rho$ largest and 
with all its remaining singular values,
respectively.
 \item%5
Our algorithms use much fewer random parameters and run much faster 
 when various random sparse and  structured preprocessors replace
Gaussian. Empirically the outputs of the resulting algorithms  
is as accurate as the outputs under Gaussian preprocessing.
  \item%6
Our  novel {\em duality techniques} provides
 formal support, so far missing, for these 
empirical observations and opens door to {\em derandomization} of 
our preprocessing and to
further acceleration and simplification of our algorithms
 by using more efficient sparse and structured preprocessors.
\item%7
Our techniques and our progress can
be applied to various other fundamental 
matrix computations
 such as the celebrated low-rank approximation of a matrix by means of
random sampling.
\end{itemize} 
\end{abstract}

% - - - - - - - - - - - - - - - - - - - - - - - - - - - - - - - - - - - - -

\paragraph{\bf 2000 Math. Subject Classification:}
65F05, 65F35, 15A06, 15A52, 15A12 

% - - - - - - - - - - - - - - - - - - - - - - - - - - - - - - - - - - - - -

\paragraph{\bf Key Words:}
Randomized matrix algorithms;
Gaussian random matrices; Singular spaces of a matrix; 
Duality; Derandomization; Sparse and structured preprocessors

% - - - - - - - - - - - - - - - - - - - - - - - - - - - - - - - - - - - - -

\section{Introduction}\label{sintro}

% - - - - - - - - - - - - - - - - - - - - - - - - - - - - - - - - - - - - -

\subsection{Randomized augmentation: outline}\label{sraugout}

% - - - - - - - - - - - - - - - - - - - - - - - - - - - - - - - - - - - - -

A standard Gaussian $m\times n$ random matrix, $G$ 
(hereafter referred to just as {\em Gaussian}),
has full rank with probability 1 (see Theorem \ref{thdgr}).
Furthermore the expected spectral norms $||G||$
and
$||G^+||$, $G^+$ denoting 
 the Moore-Penrose generalized inverse,   
satisfy the following estimates (see Theorems \ref{thsignorm} and \ref{thsiguna}): 
\begin{itemize}  
  \item%1
$\mathbb E(||G||)\approx 2\sqrt h$, for $h=\max\{m,n\}$, 
\noindent and
\item%2
$\mathbb E(||G^+||)\le \frac{e\sqrt{l}}{|m-n|}$ 
provided that $l=\min\{m,n\}$, $m\neq n$, and $e=2.71828\dots$. 
\end{itemize}
Thus, for moderate or reasonably large integers $m$ and $n$,
the matrix $G$ can be considered well-conditioned
with the confidence growing fast as the integer $|m-n|$ increases from 0.
By virtue of  part 2 of 
Theorem \ref{thsiguna},
the matrix $G$ can be viewed as  well-conditioned even  for 
$m=n$, although
 with a grain of salt, depending on context.

Motivated by this information, we append sufficiently but 
reasonably many Gaussian rows or columns
 to any matrix $A$,
possibly rank deficient or ill-conditioned,
but normalized, such that $||A||=1$.
(Our approach requires attention to various pitfalls,
and in particular it fails without normalization 
of an input matrix.)
Then we prove that the cited properties of a Gaussian matrix
also hold for the augmented matrix $K$ and similarly for 
the matrix $C=A+UV^T$ where  $U$ and $V$ are  Gaussian matrices.

We, however, prove and confirm empirically that 
 {\em randomized augmentation}  $A\rightarrow K$ above
is likely to produce matrices with smaller condition  numbers than
{\em randomized additive preprocessing} $A\rightarrow C=A+UV^T$
and than augmentation by appending to 
a matrix $A$ two blocks of rows and columns simultaneously.
These results should help direct properly 
our randomization. 

Its main application area
is the computations with rank deficient and ill-conditioned matrices.
In particular, suppose we are given a matrix $A$ that has  a 
numerical rank $\rho$ and seek approximate bases for 
we approximate closely 
the leading and trailing singular spaces        
associated with the $\rho$ largest and with all the 
remaining singular values of that matrix, respectively.

The known numerical algorithms solve these problems
by applying pivoting, orthogonalization, or the Singular Value Decomposition 
(SVD). Orthogonalization and particularly SVD are more costly (and more reliable),
but even pivoting takes its toll -- it interrupts the stream of arithmetic operations 
with foreign operations of comparison,
involves book-keeping, compromises data locality, 
increases communication overhead and data dependence, 
readily destroys matrix structure and 
sparseness, and threatens or undermines application of block matrix algorithms.

In the next two sections we solve these problems 
by applying randomized augmentation or additive preprocessing
at  a much lower randomized computational cost
versus the expensive known techniques. 

% - - - - - - - - - - - - - - - - - - - - - - - - - - - - - - - - - - - - -

\subsection{Randomized sparse and structured preprocessing}\label{srasps}

% - - - - - - - - - - - - - - - - - - - - - - - - - - - - - - - - - - - - -

Our study has some similarity with 
the celebrated work on 
{\em low-rank approximation} of a matrix by means of random 
 sampling (cf. \cite{HMT11}) and with 
randomized preprocessing of Gaussian elimination without 
pivoting\footnote{Hereafter we use the acronym {\em GENP}.} in \cite{PQY15}.  
In particular, similarly to randomized low-rank approximation
in \cite[Section 11]{HMT11},
 our techniques, algorithms and their analysis can be extended to 
the case where  preprocessing with Gaussian matrices 
is replaced by preprocessing with
{\em Semisample Random Fourier Transform}\footnote{Hereafter 
we use the acronym {\em SRFT}.} structured matrices, 
defined in our Appendix \ref{ssrft} 
and \cite[Section 11]{HMT11}.
 
The  transition from Gaussian to SRFT preprocessing greatly simplifies 
the computations, but increases
the estimated probability of failure.
This estimate, however, seems to be overly pessimistic 
for most inputs because empirical frequency of failure 
(observed consistently in our tests and in the 
 tests covered in \cite{HMT11})
was about the same in the cases of Gaussian and SRFT
 preprocessing.

More generally, our tests (as well as the tests 
for randomized low-rank approximation 
 by many authors and the tests
for GENP in  \cite{PQY15}) have  consistently produced
 similar outputs with about the same accuracy  when  preprocessing 
with various random sparse and structured matrices
(including SRFT matrices 
as a special subclass)
replaced Gaussian preprocessing (cf. Table \ref{tabprec}).

Formal support for such empirical observations
has been a challenge for quite a while, 
and our simple but novel insight enables us to provide it
finally:
we prove  that {\em the  known estimates for the impact
 of preprocessing with a Gaussian multiplier 
onto any input matrix can be extended to
preprocessing with any well-conditioned multiplier
of full rank 
onto average input matrix} and consequently 
onto a statistically typical, that is,  almost any input matrix
with a narrow class of exceptions.
In this basic {\em Duality Theorem} 
we assume that 
average matrix is defined  under the 
Gaussian probability distribution. Such a  
provision is  customary,
and it is quite natural
in view of the Central Limit Theorem.

Regarding the class of allowed  multipliers,
the restriction in the theorem is the
{\em mildest possible} and allows us to 
select sparse and structured multipliers
which can be generated and multiplied 
by an input matrix as fast as one could wish.
Thus, besides providing
 formal support, so far missing, for the cited 
empirical observations, our results open door to {\em derandomization} of 
our preprocessing and to
further {\em acceleration and simplification} of the known algorithms
by using more efficient sparse and structured preprocessors.

Our reports \cite{PZa} and \cite{PZb}
have furnished such a simple but novel duality techniques 
also for low-rank approximation and
GENP with further extension to Fast Multipole and Conjugate Gradient  
celebrated algorithms.

%------------------------------------------------------------------------------

\subsection{Some related works and further research directions}\label{sextin}

% - - - - - - - - - - - - - - - - - - - - - - - - - - - - - - - - - - - - -

Our present study
continues and enhances 
the progress in the works \cite[Section 2.13]{BP94}, \cite{PY07}, \cite{PMRT07}, 
\cite{W07}, {PIMR08a}, \cite{PIMR08b}, 
\cite{PGMQ08}, \cite{PY09},
 \cite{PIMR10}, \cite{PQ10}, \cite{PQ12}, \cite{PQY15}, \cite{PQZC}, \cite{PQZ13},
 and  \cite{PY09}
on increasing the efficiency of
matrix algorithms by means of randomized preprocessing.
Unlike these earlier works, we support the favorable
results of our extensive tests with detailed 
 formal analysis.

Our Algorithms 
3.1t  and 3.1t+ show that
the power of randomized multiplication, studied extensively in 
 \cite[Section 2.13]{BP94},
 \cite[Section 12.2]{PGMQ08},
 \cite{PY09}, \cite{HMT11}, \cite{PQZ13}, 
 \cite{PQY15},  \cite{PZ15},  \cite{PZa},  \cite{PZb}, 
and the references therein,
 can be 
enhanced when we combine it with randomized augmentation
or additive preprocessing.

The search for such synergistic combinations is a 
natural and important research  challenge.
As we have pointed out already, our work 
should motivate bolder application of 
sparse and structured preprocessing
towards simplification  and acceleration of
matrix computations. Our progress should motivate efforts for 
the extension of our techniques and results to other 
fundamental matrix computations, by following 
the first steps in these directions in
  \cite{PZa} and \cite{PZb}.

% - - - - - - - - - - - - - - - - - - - - - - - - - - - - - - - - - - - - -

\subsection{Organization of the paper}\label{sorg}

% - - - - - - - - - - - - - - - - - - - - - - - - - - - - - - - - - - - - -

We organize our paper as follows.

In the next 
subsection 
and in the Appendix we cover some definitions and auxiliary results. 
In Sections \ref{saptr} and \ref{saptrl} we  
  approximate leading and trailing singular spaces of a matrix
that has smaller numerical rank
 by applying our randomization techniques.
These two sections
make up Part I of our paper, devoted to our algorithms.

In Sections \ref{saug} and \ref{sapaug} we
 estimate the
impact of Gaussian augmentation and additive 
preprocessing on the condition number of a matrix,
these estimates imply correctness of our algorithms
of Sections \ref{saptr} and \ref{saptrl}.
In Section \ref{sweak} we extend our study 
to the case of sparse and structured randomization
and present our results on dual randomization.
Sections \ref{saug}--\ref{sweak}
form Part II of our paper, devoted to the analysis of our algorithms.

Section \ref{sexp} covers our numerical tests,
which are the contribution of the second author.
 In Section \ref{srel} we summarize our study
and discuss some directions for further research.
Sections \ref{sexp} and \ref{srel}
make up Part III of our paper, devoted to tests, summary, 
and extension of our algorithms.

%------------------------------------------------------------------------------

\subsection{Some basic definitions}\label{sdef}

%------------------------------------------------------------------------------

%$\mathbb  R^{m\times n}$ is the class of real $m\times n$ matrices $A=(a_{i,j})_{i,j}^{m,n}$.

%------------------------------------------------------------------------------

Except for Appendix \ref{ssrft}, we work in the field $\mathbb R$ of real numbers,
but a large part of our study can be extended to the computations in the field 
$\mathbb C$ of complex numbers (cf.  \cite{E88}, \cite{ES05}, \cite{CD05}).

Hereafter 
 the concepts ``large", ``small", ``near", ``close", ``approximate", 
``ill-conditioned" and ``well-conditioned" are 
quantified in the context. By saying ``likely" 
we mean with a probability close to 1.
  
$(B_1~|~\dots~|~B_k)=(B_j)_{j=1}^k$ 
 denotes a $1\times k$ block matrix with the blocks $B_1,\dots,B_k$.

$I$ and $I_{k}$ denote the $k\times k$  identity matrix.

$O$ and $O_{k,l}$ denote the $k\times l$ matrix filled with zeros.

 $||M||=||M||_2$ is the spectral norm of a matrix $M$.

%-------------------------------------------------------------------------------
  
For a matrix $M$ having full column rank,
 $Q(M)$ denotes a unique orthogonal matrix 
defined by the  
QR factorization $M=QR$  
where $R=R(M)$ is a unique 
 upper triangular square matrix 
with positive diagonal entries 
(cf. \cite[Theorem 5.2.3]{GL13}).

%------------------------------------------------------------------------------

$\mathcal  G^{m\times n}$ is the class of Gaussian $m\times n$ matrices.

See some additional definitions 
in Section \ref{sbss} 
and the Appendix. 

\medskip
 
\medskip

\medskip

%------------------------------------------------------------------------------

{\Large \bf \em PART I: Randomized Matrix Algorithms}

%------------------------------------------------------------------------------
%------------------------------------------------------------------------------

\section{Approximation of the Leading Singular Spaces}\label{saptr}

%------------------------------------------------------------------------------
%------------------------------------------------------------------------------

\subsection{Left 
%and right 
inverses, matrix bases, nmbs,  and
singular spaces}\label{sbss}

%------------------------------------------------------------------------------

An $m\times n$  matrix $M$ has an $n\times m$ 
{\em left 
inverse} matrix $X=M^{(I)}$ such that  $XM=I_n$
 if and only  if it has full column rank $n$. 
(We 
can  
compute 
at first  
 QR factorization $M=QR$
for orthogonal $m\times n$
matrix $Q$ and then 
a left inverse $M^{(I)}=R^{-1}Q^T$,
by performing $O(mn^2)$ flops
overall.)

A matrix having full column rank is a
{\em matrix basis} for its range.
A matrix basis $B$ for
the null 
space
 $\mathcal N(M)$
is a {\em null matrix basis} or a {\em nmb} 
for the matrix $M$, denoted $\nmb(M)$.
In  other words  $B=\nmb(M)$
if the matrix $B$ has full column rank and if 
%\begin{equation}\label{eqnmbrn} 
$\mathcal R(B) =\mathcal N(M)$.

Suppose that we are given three integers $k$, $m$ and $n$, $1<k<\min\{m,n\}$, 
an $m\times n$ matrix $M$ of rank $\rho$, and its 
{\em  SVD} 
\begin{equation}\label{eqsvd}
M=S_M\Sigma_MT_M^T,
\end{equation}
where  $S_M$ and  $T_M$ are
 square orthogonal matrices,
$\Sigma_M=\diag(\widehat \Sigma_M,O_{m-\rho,n-\rho})$ 
is the diagonal matrix of the singular values,
$$\widehat \Sigma_M=\diag(\sigma_j(M))_{j=1}^{\rho},~\sigma_1=||M||,~{\rm and}~ 
\sigma_1\ge \sigma_2\ge \cdots \ge \sigma_{\rho}>0.$$ 

 Partition the matrices $S_M$,
$\Sigma_M$,
and $T_M$  into  their leading and trailing parts as follows,
\begin{equation}\label{eqsvdpart}
S_M=(S_{k,M}~|~S_{M,k}),~\Sigma_M=\diag(\Sigma_{k,M},\Sigma_{M,k}),
~{\rm and}~T_M=(T_{k,M}~|~T_{M,k}),
\end{equation}
where $S_{k,M}\in \mathbb R^{m\times k}$,
$T_{M,k}\in \mathbb R^{n\times k}$, $S_{M,k}\in \mathbb R^{m\times (m-k)}$, 
$T_{M,k}\in \mathbb R^{n\times (n-k)}$,
 $\Sigma_{k,M}=\diag(\sigma_j(M))_{j=1}^k$, and
$\Sigma_{M,k}=\diag(\diag(\sigma_j(M))_{j=k+1}^{\rho},O_{m-\rho,n-\rho})$.

Now write $\mathbb S_{k,M}=\mathcal R(S_{k,M})$,
$\mathbb T_{k,M}=\mathcal R(T_{k,M})$,
$\mathbb S_{M,k}=\mathcal R(S_{M,k})$, and
$\mathbb T_{M,k}=\mathcal R(T_{M,k})$.

If $\sigma_k>\sigma_{k+1}$, 
then 
$\mathbb S_{k,M}$ and $\mathbb T_{k,M}$ are
the leading left and right singular spaces
    associated with the $k$ largest 
singular values of the matrix $M$, respectively,
and $\mathbb S_{M,k}$,  and $\mathbb T_{M,k}$ are
the trailing left and right singular spaces
    associated with the remaining 
singular values, respectively.

For $k=\rho$, we arrive at  {\em compact SVD},
$M=S_{\rho,M}\Sigma_{\rho,M}T_{\rho,M}^T$ where $\Sigma_{\rho,M}=\widehat \Sigma_M$.

%$||A||=||A||_2$ denotes the spectral norm of the matrix $A$.

For a positive tolerance $\eta$, a  matrix $M$ has 
$\eta$-rank $ \rho$,  $\rho=\rank_{\eta}(M)$,
if  $\sigma_{ \rho}(M)<\eta\le\sigma_{ \rho+1}(M)$
or, equivalently, if the matrix $M$ can be approximated 
within the norm bound $\eta$ by a matrix
of rank $ \rho$, but not by a matrix of rank $ \rho-1$.
Note that
$$\rank_{\eta}(M)\le\rank_{\eta'}(M)\le \rank (M)~{\rm if}~\eta\ge\eta'.$$
 
 $\eta$-rank is said to be {\em numerical rank}
 if $\eta$ is small (in context).

%------------------------------------------------------------------------------

\subsection{Linking approximation of a matrix and of its leading singular space}\label{smtrldng}

%------------------------------------------------------------------------------

For
 an $m\times n$ matrix $A$  having
 numerical rank $\rho$,  
seek approximation to its leading singular space  
$\mathbb T_{\rho,A}$. 
The following theorem links closely this task to
the celebrated task 
of low-rank approximation of a matrix $A$,
extensively covered in \cite{HMT11}.

%------------------------------------------------------------------------------

\begin{theorem}\label{thlrappr} 
Let $\nrank(A)=\rho$. Write $\Delta=Q-T_{\rho,A}V$
for  a $\rho\times \rho_+$ orthogonal matrix $V$
where $\rho_+\ge \rho$. Then 
$$\frac{||AQQ^T-A||}{||A||}\le (2+||\Delta||)||\Delta||+\frac{\sigma_{\rho+1}(A)}{\sigma_{1}(A)}.$$ 
\end{theorem}
\begin{proof} 
Deduce from the equation $T_{A}^TT_{\rho,A}=(I_{\rho}~|~O_{n-\rho,\rho})^T$
 that  
$$AT_{\rho,A}T_{\rho,A}^T=
S_{A}\Sigma_{A}T_A^TT_{\rho,A}T_{\rho,A}^T=
S_{\rho,A}\Sigma_{\rho,A}T_{\rho,A}^T=A_{\rho}.$$

Recall that $Q=T_{\rho,A}V+\Delta$,
 $T_{A}^TT_{\rho,A}=(I_{\rho}~|~O_{n-\rho,\rho})^T$,
 and $A=A_{\rho}+\bar A_{\rho}$
where \\
 $\bar A_{\rho}=S_{\rho,A}\Sigma_{\rho,A}T_{\rho,A}^T$
 and $||\bar A_{\rho}||\le \sigma_{\rho+1}(A)$.
Combine 
  the above equations  
and obtain $$AQQ^T-A=
%S_{\rho,A}\Sigma_{\rho,A}T_{\rho,A}^T(T_{\rho,A}+\Delta)(T_{\rho,A}^T+\Delta^T)+E'=
-\bar A_{\rho}+AT_{\rho,A}V\Delta^T+A\Delta (V^TT_{\rho,A}^T+\Delta^T).$$
 
Now substitute
$||\bar A_{\rho}||=\sigma_{\rho+1}(A)$ and $||T_{\rho,A}||=||V||=1$
and obtain
$$||AQQ^T-A||\le \sigma_{\rho+1}(A)+(2+||\Delta||)||\Delta||~||A||.$$
The theorem follows because
$||A||=\sigma_{1}(A)$. 
\end{proof} 

%------------------------------------------------------------------------------

\begin{remark}\label{rertsvd}  
If the error norm $||\Delta||$ of the 
approximation to
the leading singular space $\mathbb T_{\rho,A}$
 is  small, then, by virtue of Theorem \ref{thlrappr}, 
  the relative error 
of rank-$\rho$ approximation of the matrix $A$ by 
$AQQ^T$ is also small.
Conversely,  
if the ratio $\frac{||AQQ^T-A||}{||A||}$ is small, then by applying  
 \cite[Algorithm 5.1]{HMT11} one can 
approximate the matrices $S_{\rho,A}\approx QS_{Q^TA}$ and $T_{\rho,A}\approx QS_{AQ}$
of the leading singular vectors 
essentially at the cost of computing compact SVDs of the matrices $Q^TA$
 and $AQ$ of smaller sizes. Having the matrix $S_{\rho,A}$ approximated,
we can readily approximate at first the matrix $\Sigma_{\rho,A}T^T_{\rho,A}=S_{\rho,A}^TA$
and then the matrices $\Sigma_{\rho,A}$ and $T_{\rho,A}$
(thus approximating the leading part of SVD of the matrix $A$), and similarly if we are given 
an approximation of the matrix $T_{\rho,A}$. Based on these observations,
 \cite[Section 10.2]{HMT11} readily extends \cite[Algorithm 4.1]{HMT11} 
to {\rm randomized computation of the numerical rank of a matrix}.
\end{remark}

%------------------------------------------------------------------------------

\subsection{Randomized approximation of a leading singular space}\label{sldng}

%------------------------------------------------------------------------------

\begin{definition}\label{defenu} 
$\mathbb E(v)$ denotes the expected value of 
a random variable $v$.
$\nu_{m,n}$, $\nu_{F,m,n}$, $\nu_{m,n}^+$, and $\kappa_{m,n}$ denote the 
 random variables
 $||G||$, $||G||_F$ (Frobenius norm of $G$), 
 $||G^+||$, and $\kappa(G)=||G||~||G^+||$, respectively,
and $\nu_{n}^+=\nu^+_n(A)$ denote the norm $||(A+G)^{+}||$
provided that $A\in \mathbb  R^{n\times n}$ and
$G\in \mathcal  G^{n\times n}$. 
\end{definition}
Note that $\nu_{n,m}=\nu_{m,n}$,
$\nu_{n,m}^+=\nu_{m,n}^+$, and $\kappa_{n,m}=\kappa_{m,n}$,
and assume that the random variables $\nu_{m,n}$,  $\nu_{m,n}^+$,  $\nu_{F,m,n}$, 
and $\nu_{n}^+$ turn into 1 if $m=0$ or $n=0$.

By virtue of
\cite[Theorem 3.3]{SST06}, for $n\ge 2$, a real $x>0$, and a matrix 
$A\in \mathbb  R^{n\times n}$, it holds that
\begin{equation}\label{eqsst06}
  {\rm Probability}~\{\nu_{n}^+\ge x\}\le 2.35 {\sqrt n}/x.
\end{equation} 

Our next task is the  approximation of a leading singular space $\mathbb T_{\rho, A}$
of  a matrix $A$ that has numerical rank $\rho$.
We adopt the technique of {\em random sampling}, that is, 
approximate $\mathbb T_{\rho, A}$ by the range of the matrix $A^TH$ for a
Gaussian $m\times \rho_+$ matrix $H$ and for a nonnegative 
but not large integer $\rho_+-\rho$.
This technique has been
studied in \cite{HMT11}
for low-rank approximation of such a matrix $A$, but 
our error analysis is a little different because 
we approximate the space $\mathbb T_{\rho, A}$ rather than the matrix $A$.
The following theorem estimates the approximation error.

%------------------------------------------------------------------------------

\begin{theorem}\label{thtrfrld} (Cf. Definition \ref{defenu}.)
Suppose that
 an $m\times n$ matrix $A$  has
 numerical rank $\rho$,
$H$ is an $n\times \rho_+$ Gaussian matrix, 
$H\in \mathcal G^{n\times \rho_+}$,
and $m\ge n\ge \rho_+\ge \rho>0$.
Then with probability 1 there exists an 
$n\times \rho_+$ matrix $X$ of rank $\rho$
such that  
$\mathcal R(X)=\mathbb T_{\rho,A}$
and  $||A^TH-X||\le\sigma_{\rho+1}(A)\nu_{n,\rho_+}$.
\end{theorem}
\begin{proof}
Recall equations (\ref{eqsvd}) and (\ref{eqsvdpart}), for $M=A$ and
$k=\rho= \nrank (A)$, and write

 $$A^TH=A^T_{\rho}H+\bar A^T_{\rho}H,~ 
 \bar A_{\rho}=S_{A,\rho}\Sigma_{A,\rho}T_{A,\rho}^T,~{\rm and}~
A_{\rho}=S_{\rho,A}\Sigma_{\rho,A}T_{\rho,A}^T.
$$
Then 
\begin{equation}\label{edtrnrm}
 A^T_{\rho}H=T_{\rho,A}\Sigma_{\rho,A}B~{\rm and}~
||\bar A^T_{\rho}H||\le ||\bar A^T_{\rho}||~||H||=\sigma_{\rho+1}(A)\nu_{n,\rho_+}
\end{equation}
 where
$B=S^T_{\rho,A}H$ is a $\rho \times \rho_+$
Gaussian  matrix by virtue of Lemma \ref{lepr3}.

   Now the theorem follows for $X=A^T_{\rho}H$ because
the matrix $\Sigma_{\rho,A}$ is nonsingular by assumption,
and with probability 1 the matrix $B$ has full rank,  
by virtue of Theorem \ref{thdgr}.
\end{proof}

The bound $\sigma_{\rho+1}(A)\nu_{n,\rho_+}$ 
of (\ref{edtrnrm})
can be large only with a  probability close to 0, 
and one can monitor the  approximation error 
 by estimating the ratio $\frac{||AQQ^T-A||}{||A||}$
(see Remark \ref{rertsvd}). 
Probabilistic estimates for this ratio
in \cite[Sections 10.2 and 10.3]{HMT11}
have order $\sigma_{\rho+1}(A)$ and hold with
 a probability  $1-3/p^{p}$ for an oversampling integer 
$p=\rho_+-\rho$ if $p\ge 20$. 

%------------------------------------------------------------------------------
 
Theorem
\ref{thtrfrld} implies correctness of the following simple randomized algorithm,
which is a subalgorithm of \cite[Algorithm 4.1]{HMT11}.

%------------------------------------------------------------------------------

\begin{algorithm}\label{algldbs} (Cf. Remarks \ref{regap} and \ref{relfts}.)

%------------------------------------------------------------------------------

\begin{description}

%------------------------------------------------------------------------------

\item[{\sc Input:}] 
Three integers $m$, $n$, and $\rho_+$ 
such that $m\ge n\ge \rho_+>0$, and an
$m\times n$ matrix  $A$ having numerical rank $\rho\le \rho_+$.
%and a positive tolerance value $t$.

%------------------------------------------------------------------------------

\item[{\sc Output:}] 
% FAILURE with a probability close to 0 or 
An orthogonal $n\times \rho_+$ matrix $X$, whose range 
is likely to approximate
 the leading singular space $\mathbb T_{\rho,A}$.

%------------------------------------------------------------------------------

\item[{\sc Computations}:]

\begin{enumerate}
\item %1
Generate a  Gaussian $n\times \rho_+$ matrix $H$.
\item %2
Compute and output the  $n\times \rho_+$ matrix $X=A^TH$.
%\item %3
%Compute  the   $n\times \rho_+$ orthogonal factor $V$
% in a UTV factorization  $A^TH=UTV^T$  
%(see \cite[Section 5.4.6]{GL13} and \cite[Section 5.4]{S98} for these factorizations). 

%Output  the $n\times \rho$ leftmost submatrix $X$ of this factor.
\end{enumerate}

%------------------------------------------------------------------------------

\end{description}

%------------------------------------------------------------------------------

\end{algorithm}

%------------------------------------------------------------------------------
 
The algorithm 
%\ref{algldbs}
 generates $n\rho_+$ i.i.d. Gaussian values
and then performs $(2n-1)m\rho_+$ flops, but
we need only $n+\rho_+$ random parameters 
and $O(mn\log(\rho_+)+n\rho_+^2)$ flops
if we replace the $n\times \rho_+$ Gaussian multiplier $H$ with an
$n\times \rho_+$ 
SRFT 
structured multiplier. Hereafter we
 refer to Algorithm \ref{algldbs} with a SRFT multiplier
as {\bf Algorithm 3.1+}. 
 Then again we can monitor its output error norm 
by estimating the ratio  $\frac{||AQQ^T-A||}{||A||}$.
According to the study
of SRFT multipliers
 in  \cite[Section 11]{HMT11}, 
the ratio is large with
a probability in $O(1/r))$ if 
  $\rho_+$ has order $(\rho+\log(n))\log (\rho)$, but
 empirically even the choice of  $\rho_+=\rho+20$ 
``is  adequate in almost all applications".

%------------------------------------------------------------------------------

\begin{remark}\label{regap} (Cf. \cite[Theorem 9.2]{HMT11}.)
The approximation of a basis for the leading (as well as trailing) singular spaces  
%(see Section \ref{stlss1})
is facilitated
as the gaps increase between the singular
values of the input matrix $A$.
This motivates preprocessing of an input matrix $A$ by means of
the power transforms $A\Longrightarrow B_h=(AA^T)^hA$
 for positive integers $h$
because $\sigma_j(B_h)=(\sigma_j(A))^{2h+1}$ for all $j$.
\end{remark}

%------------------------------------------------------------------------------

\begin{remark}\label{relfts} 
By applying the algorithms of this subsection to the transpose $A^T$ 
we can approximate the left singular spaces of our input matrix $A$.
If, however, an approximation $Q_T=T_{\rho,A}V+\Delta_T$
to a matrix basis for the right singular space $\mathbb T_{\rho,A}$
is  already available,
then we can readily compute an approximation $AQ_T$ to the matrix basis
for the left singular space $\mathbb S_{\rho,A}$.
Indeed $AQ_T=S_A\Sigma_A T^T_A(T_{\rho,A}V+\Delta_T)=S_A\Sigma_A T^T_AT_{\rho,A}V+
A\Delta_T=S_{\rho,A}\Sigma_{\rho,A}V+A\Delta_T$, and 
so the matrix $Q_S=AQ_T$ is an approximate matrix basis $S_{\rho,A}U$ for 
the left singular space $\mathbb S_{\rho,A}$ within the error norm bound
$||\Delta_S||\le ||A||~||\Delta_T||$.
Furthermore we can compute the matrix $Q_S^TAQ_T=U^T\Sigma_{\rho,A}V+\Delta_{\Sigma}$
where $\Delta_{\Sigma}=\Delta_S^TAQ_T+Q_S^TA\Delta_T-\Delta_S^TA\Delta_T$,
and so $\frac{||\Delta_{\Sigma}||}{||A||}\le ||\Delta_S||+||\Delta_T||+||\Delta_S||||\Delta_T||$.
Then the singular values of the $\rho\times \rho$ matrix
$Q_S^TAQ_T$ approximate those of the matrix $A$.
\end{remark}

%------------------------------------------------------------------------------

\subsection{Oversampling and compression}\label{sovercmpr}

%------------------------------------------------------------------------------

If we know numerical rank $\rho$ of the input matrix $A$, we can 
apply Algorithm \ref{algldbs} or 3.1+, for $\rho_+=\rho$.
Otherwise we can compute $\rho$ by applying Algorithm \ref{algldbs} 
or 3.1+ in a binary search
process. Indeed, let $X$ denote the output matrix of the algorithm. Then
 the norm $||AQQ^T-A||$ has order of $\sigma_{\rho+1}(A)$
for $Q=Q(X)$ if $\rho_+\ge \rho$
(cf. Theorems \ref{thlrappr} and \ref{thtrfrld}), but 
is at least $\sigma_{\rho}(A)$
if $\rho_+<\rho$. 

Alternatively, having applied  Algorithm \ref{algldbs} or 3.1+,  for $\rho_+>\rho$,
 we can compress the $n\times \rho_+$ output matrix $X$
into $n\times \rho$ orthogonal matrix by means of computing a rank-revealing 
QR factorization, a UTV  factorization, or SVD of the matrix $X$
(see \cite[Section 5.4]{GL13} and \cite[Section 5.4]{S98} for these factorizations).
Such computations are relatively inexpensive if $\rho_+\ll \min\{m,n\}$,
and are routine in the extension of the 
algorithm to low-rank approximation of the matrix $A$.

For the task of the approximation of the singular space 
 $\mathbb T_{\rho,A}$, however, estimated approximation error norms
of these  computations
are a little larger than $\nu_{n,\rho_+}\sigma_{\rho+1}$, even where we compute 
the matrix $S_{\rho,X}$ of the $\rho$ leading left singular vectors 
of the matrix $X=A^TH$ and  output it
as an approximate matrix basis for the space $\mathbb T_{\rho,A}$. 
Here are some relevant estimates. 

%------------------------------------------------------------------------------

\begin{theorem}\label{thldld} 
Under the assumptions of Theorem
\ref{thtrfrld}, write
\begin{equation}\label{edphi} 
 \phi=\sqrt{n-\rho}~\sigma_{\rho+1}(A)\nu_{F,n,\rho_+}||(A^T_{\rho}H)^+||.
\end{equation}
Let the matrix $A^T_{\rho}H$ have full rank $\rho$ and let
$\phi\ge 1$.
Then the $n\times \rho$ orthogonal matrix $S_{A^T_{\rho}H}$ of the $\rho$  
leading left singular vectors of the matrix $A^T_{\rho}H$ 
approximates a matrix basis of the leading singular space 
$\mathbb T_{\rho,A}$ of the matrix $A$ within the Frobenius error norm 
$4\phi$.
\end{theorem}
\begin{proof}
Equation (\ref{edtrnrm}) implies that $\mathcal R(T_{\rho,A})=\mathcal R(S_{A_{\rho}^TH})$.
Recall  that $A^TH=A^T_{\rho}H+\bar A^T_{\rho}H$ and combine 
the upper bound (\ref{edtrnrm}) on the norm $||\bar A^T_{\rho}H||$
 with \cite[Theorem 8.6.5]{GL13} where $E=\bar A^T_{\rho}H$ and 
$A$ is replaced by $A^T_{\rho}H$
(which implies that $\delta=\frac{1}{||(A^T_{\rho}H)^+||}$ in that theorem).
\end{proof}

%------------------------------------------------------------------------------

\begin{theorem}\label{thldld1} 
Under the assumptions of Theorem
\ref{thtrfrld}, it holds that
$$||(A^T_{\rho}H)^+||\le \frac{\nu^+_{\rho, \rho_+}}{\sigma_{\rho}(A)}.$$
\end{theorem}

%------------------------------------------------------------------------------

\begin{proof}
Recall that $A^T_{\rho}H=T_{\rho,A}\Sigma_{\rho,A}B$ for $B\in \mathcal G^{\rho\times \rho_+}$
(cf. (\ref{edtrnrm})).

Write $F=\Sigma_{\rho,A}B$  and let $F=S_F\Sigma_FT_F^T$
and $B=S_B\Sigma_BT_B^T$
be compact SVDs. 

$T_{\rho,A}S_F$ is an orthogonal matrix because
$S_F$ is an $\rho \times \rho$ orthogonal matrix. 

Now write  $S_{A^T_{\rho}H}=T_{\rho,A}S_F$ and note that
$A^T_{\rho}H=S_{A^T_{\rho}H}\Sigma_FT_F^T$ 
is a  compact SVD.

Consequently 
$||(A^T_{\rho}H)^+||=||F^+||$.

Furthermore $F=\Sigma_{\rho,A}S_B\Sigma_BT_B^T$
where $\Sigma_{\rho,A}$, $S_B$, and $\Sigma_B$
are $\rho \times \rho$ nonsingular matrices. 

Therefore $F^+=
T_B\Sigma_B^{-1}S_B^{T}\Sigma_{\rho,A}^{-1}$, 
where  
$||S_B||=||T_B||$=1.

It follows that
$||(A^T_{\rho}H)^+||=||F^+||\le||\Sigma_B^{-1}||~||\Sigma_{\rho,A}^{-1}||=
\frac{\nu^+_{\rho, \rho_+}}{\sigma_{\rho}(A)}$.
\end{proof}

The latter two theorems together imply the following corollary.

%------------------------------------------------------------------------------

\begin{corollary}\label{coldld} 
Under the assumption of Theorem
\ref{thtrfrld}, 
write
\begin{equation}\label{edphi+} 
\phi_+=\sqrt{n-\rho}~\nu_{F,n,\rho_+}\nu^+_{\rho,\rho_+}\frac{\sigma_{\rho+1}(A)}{\sigma_{\rho}(A)}. 
\end{equation}
Then, with a probability at least 
${\rm Probability}\{5\phi_+\le 1\},$ 
the $n\times \rho$  matrix $S_{A^T_{\rho}H}$ of the $\rho$ leading  
left singular vectors of the matrix $A^T_{\rho}H$ 
approximates
 a matrix basis of the leading singular space 
$\mathbb T_{\rho,A}$ of the matrix $A$ within the Frobenius error norm $4\phi_+$. 
\end{corollary}

%------------------------------------------------------------------------------

\begin{remark}\label{rephi+e} 
For $\rho_+>\rho$, equation (\ref{edphi+}) and Theorems \ref{thsignorm} and  \ref{thsiguna} together imply that

%\begin{equation}\label{edephi+} 
$$\mathbb E(\phi_+)< e~
(1+\sqrt n+\sqrt{\rho_+})~\sqrt{n-\rho}~\frac{\sigma_{\rho+1}(A)}{\sigma_{\rho}(A)}\frac{\sqrt{n\rho_+
\rho}}{\rho_+-\rho},~{\rm for}~e=2.71282\dots.$$
%\end{equation}
\end{remark}

%------------------------------------------------------------------------------

\subsection{Leading singular spaces via the maximum volume}\label{smxv}

%------------------------------------------------------------------------------

One can alternatively 
approximate
%a basis for
 leading singular spaces by applying
 the algorithm
of \cite{GOSTZ10}, devised for the approximation of 
the so called CUR
decomposition of a matrix. 
The algorithm is  heuristic, 
but consistently converges  fast according to
its extensive tests by the authors.

It accesses only 
a small fraction of the entries of the input matrix.
This makes it particularly efficient for sparse matrices. 
The algorithm interchanges rows and columns of 
an input matrix, destroying Toeplitz-like, Hankel-like,
and even Van\-der\-monde-like matrix structures,
but one can 
%circumvent 
fix this deficiency by means of the back and forth
transition to Cauchy-like matrices  \cite{P15},
whose structure is invariant in row and column 
interchange.

The algorithm relies on the following  result
where we write 
$v_\rho(M)=\max_X|\det (X)|$  with the maximum 
over all $\rho\times \rho$ submatrices $X$
of a matrix $M$, and we call $v_\rho(M)$ the {\em maximal volume} of
all $\rho\times \rho$ submatrices of the matrix $M$. 

%------------------------------------------------------------------------------

\begin{theorem}\label{thgt}  \cite[Corollary 2.3]{GT01}.
Let an $n\times m$ matrix 
$A^T=\begin{pmatrix} A_{11} & A_{12} \\
 A_{21} & A_{22}
\end{pmatrix}$
have
a nonsingular $\rho\times \rho$  leading block $A_{11}$.
Write  $\nu=\frac{v_\rho(A)}{|\det A_{11}|}$, 
$C=\begin{pmatrix} A_{11}  \\
 A_{21} 
\end{pmatrix}$, and
$R=(A_{11}~|~ A_{12})$ and let
$||\cdot||_C$ denote the element-wise (Chebyshev) norm,
$||M||\le \sqrt{mn}~||M||_C$ for a matrix $M\in \mathbb R^{m\times n}$.
Then $$||A-CA_{11}^{-1}R||_C\le(\rho+1)\sigma_{\rho+1}(A))\nu.$$
\end{theorem}

By virtue of the theorem, the rank-$\rho$ matrix $CA_{11}^{-1}R$ approximates 
the matrix $A$ within a factor 
of $(\rho+1)\nu~\sqrt{mn}$ from
the optimal error bound $\sigma_{\rho+1}(A))$. 
($CA_{11}^{-1}R$ is a CUR decomposition if the
matrices $A_{11}$ and $U=A_{11}^{-1}$ are unitary.)

In the authors' tests, the iterative algorithm of \cite{GOSTZ10} 
has consistently produced $\rho\times \rho$ submatrices 
of the matrix $A$ that have 
 reasonably bounded ratios $\nu$. 
This work is linked to our study
because a nearly optimal rank-$\rho$ approximation  $CA_{11}^{-1}R$
to the matrix $A$ induces close approximations by the matrices $C$ and $CA_{11}^{-1}$
to $n\times \rho$ matrix bases of the leading singular space $\mathbb T_{\rho,A^T}$. 

%------------------------------------------------------------------------------

\section{Approximation of the Trailing Singular Spaces}\label{saptrl}

%------------------------------------------------------------------------------

\subsection{The basic theorems}\label{sthm}

%------------------------------------------------------------------------------
 
The following results from 
\cite{PQ10} and \cite{PQ12}
are basic for
 the   
approximation of the
trailing singular space 
$\mathbb T_{A,\rho}$.
We assume that we have already computed the numerical rank $\rho$, e.g.,  
 by applying Algorithms \ref{algldbs} or 3.1+
(cf. Remark \ref{rertsvd}). 

%$\mathcal R(Y)=\mathcal N(A)$.
%We can turn such matrix into a nmb$(A)$
%by applying a rank revealing factorization. 

%------------------------------------------------------------------------------

\begin{theorem}\label{thn}
Suppose that
 $A \in \mathbb{R}^{m \times n}$,  
$V\in \mathbb{R}^{n\times s}$, 
$\widehat K=
\begin{pmatrix}
   V^T   \\
   A
\end{pmatrix}$,
%\in  \mathbb{R}^{(m+s) \times (n+q)}$,
$\rank (V)=s$,
$\rank (\widehat K)=n$,
% and 
$m\ge n$.
Write $\widehat Y=\widehat K^{(I)} 
\begin{pmatrix}
I_{s}  \\
O_{m,s}
\end{pmatrix}$.  Then

(a) $\mathcal N(A) \subseteq \mathcal R(\widehat Y)$,

(b) $\mathcal N(A)= \mathcal R(\widehat Y)$ if $s+\rank (A)=n$,

(c) $\mathcal N(A)= \mathcal R(\widehat Y \widehat Z)$ if 
 $\mathcal R(\widehat Z)=\mathcal N(A\widehat Y)$.
\end{theorem}

\begin{proof}
See \cite[Correctness proof of Algorithm 6.1]{PQ12}.
\end{proof}

%------------------------------------------------------------------------------

\begin{theorem}\label{thnwe}
Suppose that
 $A \in \mathbb{R}^{m \times n}$,  
$U\in \mathbb{R}^{m\times q}$, 
$V\in \mathbb{R}^{n\times s}$, 
$W\in \mathbb{R}^{s\times q}$, $K=
\begin{pmatrix}
W   &   V^T   \\
U   &   A
\end{pmatrix}$,
%\in  \mathbb{R}^{(m+s) \times (n+q)}$,
$\rank(W)=q\ge \nul (A)$, $\rank (K)=n+q$,
% and 
$m\ge n$.
Write $\bar Y=(O_{n,q}~|~I_n)K^{(I)} 
\begin{pmatrix}
O_{s,q}  \\
U
\end{pmatrix}$.  
 Then 

(a) $\mathcal N(A) \subseteq \mathcal R(\bar Y)$,

(b)  $\mathcal N(A)=\mathcal R(\bar Y)$ if $\rank (U)+\rank(A)=n$,

(c)  $\mathcal N(A)=\mathcal R(\bar Y\bar Z)$
if $\mathcal R(\bar Z)=\mathcal N(A\bar Y)$.
\end{theorem}

\begin{proof}
See \cite[Theorems 11.2 and  11.3]{PQ12}.
\end{proof}

%------------------------------------------------------------------------------

\begin{theorem}\label{thnmb1} \cite[Theorem 3.1 and Corollary 3.1]{PQ10}.
Suppose 
a matrix $A\in \mathbb R^{m\times n}$ has rank $\rho$, 
$U\in \mathbb R^{m\times r}$, $V\in \mathbb R^{n\times r}$,
and 
the matrix $C=A+UV^T$  
has full rank $n$. 
Write $Y=C^{(I)}U$. 
 Then 

(a) $\mathcal N(A) \subseteq \mathcal R(Y)$ and $r\ge n-\rho$,   

(b) $\mathcal N(A)=\mathcal R(Y)$ if $r+\rho=n$,

(c) $ \mathcal N(A)=\mathcal R(YZ)$ if $\mathcal R(Z)=\mathcal N(AY)$.
\end{theorem}

%------------------------------------------------------------------------------

\begin{remark}\label{retrail}
Given a matrix $A$ and its numerical rank $\rho$, 
set to zero all but the $\rho$ largest singular values of the matrix $A$
and arrive at a  matrix $A-E$ of rank $\rho$ such that
$||E||=\sigma_{\rho+1}(A)$.
 By virtue of our next theorem, the matrix 
$T_{A,\rho}$ approximates a nmb of a matrix $A-E$ of rank $\rho$
within the norm in $O(\sigma_{\rho+1}(A))$. Therefore 
 nmb$(A-E)$ can serve as an approximate basis for the trailing singular space 
$\mathbb T_{A,\rho}$, and  
we can approximate a basis for $\mathbb T_{A,\rho}$
within $O(\sigma_{\rho+1}(A))$
by applying
the expressions of Theorems 
\ref{thn}--\ref{thnmb1} 
to the matrix $A$ rather than  to $A-E$
as long as the auxiliary matrices $\widehat K$, $K$ and $C$ in these
theorems (i) have full rank  and (ii) are well-conditioned. 
For Gaussian matrices $U$, $V$, and $W$,
property (i) above follows with probability 1 by virtue of Theorem \ref{thdgr}, 
and in Sections \ref{saug}--\ref{sweak} we specify  
our probability bounds close to 1
with which property (ii) holds.
\end{remark}

%------------------------------------------------------------------------------

\begin{theorem}\label{thtrail}
Suppose that $m\ge n$,
an $m\times n$ matrix $A$ has 
numerical rank $\rho=n-r$, and 
the matrices
$C$, $K$, $\widehat K$ of Theorems 
\ref{thn}--\ref{thnmb1}
 have full rank and are
 well-conditioned.
Define the matrices
 $Y$, $\bar Y$, and $\widehat Y$
by the expressions of 
 Theorems 
\ref{thn}--\ref{thnmb1}.
Then 
%$\rho\ge n-r_+$ and
there exist three  orthogonal $r\times r$ matrices $X$,
$Z$, and $\widehat Z$
and a scalar $c$ 
independent of $A$, $U$, $V$, $W$, $m$, $n$ and $\rho$ such that

(i) $||Q(Y)X-T_{A,\rho}||\le c\sigma_{\rho+1}(A)||U||$,

(ii) $||Q(\bar Y)Z-T_{A,\rho}||\le c\sigma_{\rho+1}(A)||\bar Y||$,

(iii) $||Q(\widehat Y)\widehat Z-T_{A,\rho}||\le c\sigma_{\rho+1}(A)||\widehat Y||$.
\end{theorem}

%------------------------------------------------------------------------------

\begin{proof} 
%Cf. \cite[Section 7.1]{PQ10}.
%The theorem turns into Theorems 
%\ref{thn}--\ref{thnmb1} if $\rho=\nrank(A)=\rank(A)$.
Apply  Theorem \ref{thnmb1} to the matrix $A-E$ of Remark \ref{retrail}
such that $rank(A-E)=\rho$ and $||E||=\sigma_{\rho+1}(A)$.
 Deduce 
that  $T_{A-E,\rho}=Q((C-E)^+U)X$, for $C=A+UV^T$ and 
 an orthogonal $r\times r$ matrix $X$,
and note that the norm $||(C-E)^+||$ is not large 
because  the matrix $C$ has full rank and is well-conditioned. 
In order to prove part (i), it remains to deduce from Theorem \ref{thpertq} 
that  $||Q((C-E)^+U)-Q(C^+U)||=O(\sigma_{\rho+1}(A)||U||)$ and
$||T_{A-E,\rho}-T_{A,\rho}Q||=O(\sigma_{\rho+1}(A))$.

Similarly we prove parts (ii) and (iii).
\end{proof}

%------------------------------------------------------------------------------

\subsection{Randomized approximation of a trailing singular space}\label{salg}

%------------------------------------------------------------------------------
%------------------------------------------------------------------------------

Assume that  $m\ge n$ and we are given an $m\times n$ matrix $A$
  and its
 numerical rank $\rho=n-r$ and seek an
 approximate  basis for the trailing singular space
$\mathbb T_{A,\rho}$. This can be also viewed as  
the search for approximate solution of the homogeneous linear system 
$A{\bf z}={\bf 0}$).
 
We can compute at first 
 an approximate matrix basis $B$ for the leading singular space 
 $\mathbb T_{A,\rho}$, by applying randomized Algorithm  \ref{algldbs}
 or 3.1+ (which involve $n\rho_+$ random parameters and $(2n-1)\rho_+$
flops or $n+\rho_+$ parameters and $O(mn\log(\rho_++n\rho_+^2))$
flops, respectively), and then
an approximate matrix basis nmb$(B)$ for the trailing singular space
$\mathbb T_{A,\rho}$. 
We refer to these algorithms as {\bf Algorithms 3.1t} and {\bf 3.1t+}.
%respectively. 

At the stage of computing a nmb$(B)$, we can apply 
 the  algorithms supporting Theorems \ref{thn}--\ref{thtrail},
but in  this application to $m\times \rho$ matrix $AH$, for $\rho<m$,
they are superseded by  \cite[Algorithm 4.1]{PQ12},
which generates an
$n\times n$ Gaussian multiplier and then performs about 
$2(n+\rho)n\rho$ flops. 

If we apply a SRFT multiplier
of Appendix \ref{ssrft} instead of the Gaussian one, then
 we would generate  only $n+\rho_+$ random values 
for $\rho_+$ of order $(\rho+\log(n))\log (\rho)$ and 
 would perform
$O((\rho_+^2+ \log (n)) n)$  flops, but  
the estimated  
 failure probability would increase from 
$3/p^p$ to the order 
$1/\rho$ (see Remark \ref{refprob} and Section \ref{sweak1}).  
%In this application the SRFT 
%are square matrices and can be replaced by
%random circulant matrices,
% scaled by $\sqrt{n/\rho}$ and post-multiplied by a random permutation matrix.
%Empirically the latter randomized algorithm is expected to
%produce accurate solution even if we skip the column interchange.

Next we describe some randomized alternatives for direct approximation
of a basis for the trailing singular space $\mathbb T_{A,\rho}$, 
which rely on 
%VP: check  Corollary  \ref{cotrail}
Theorems \ref{thn}--\ref{thtrail} and Remark \ref{retrail}.  
They can fail like Algorithms  \ref{algldbs} and 3.1+
and can run into numerical problems,
but in both cases
only with a probability close to 0
(according to our estimates in Sections \ref{saug}--\ref{sweak})
and never in our extensive tests.
Moreover, we can detect the failure
% if we compute  numerical ranks of the matrices $C$, $K$ or $\widehat K$, e.g.,
by following the recipe of Remark \ref{rertsvd}.

%------------------------------------------------------------------------------

\begin{algorithm}\label{algbasap} {\bf An approximate basis for the
 trailing singular space by using randomized  
preprocessing.}

%------------------------------------------------------------------------------

\begin{description}

%------------------------------------------------------------------------------

\item[{\sc Input:}] 
A normalized matrix 
$ A\in \mathbb R^{m\times n}$ for $m\ge n$, 
its numerical rank $\rho=n-r$, possibly computed by 
 Algorithm  \ref{algldbs} or 3.1+
(cf. Remark \ref{rertsvd}),
 and a tolerance value $\tau\gg \sigma_{\rho+1}( A)$.
%(The tolerance defined by the requested output accuracy.
%In a variation of the algorithm one can
%reapply it with a decreased tolerance 
%$\tau'$ instead of outputing FAILURE
% at Stage 4.)

%------------------------------------------------------------------------------

\item[{\sc Output:}] 
An approximate matrix 
basis $B$
of the trailing singular space $\mathbb T_{A,\rho}$
within a relative error norm bound
$\tau$.

%------------------------------------------------------------------------------ 

\item{\sc Initialization}: $~$
Choose one of Theorems 
\ref{thn}--\ref{thnmb1}
and generate the auxiliary Gaussian matrices $U$, 
$U$ and $V$, or $U$, $V$, and $W$
involved into it.

%------------------------------------------------------------------------------ 

\item{\sc Computations}: $~$ 

%------------------------------------------------------------------------------

\begin{enumerate}

%------------------------------------------------------------------------------

\item %1
Compute an approximate orthogonal matrix basis $X$ for 
the trailing singular space $\mathbb T_{A,\rho}$
by setting $X=Y$, $X=\bar Y$, or $X=\widehat Y$ and using the expression 
of the selected theorem.
Compute the matrix $AX$.

%------------------------------------------------------------------------------

\item %2
Output $B=X$ and stop
if $|| A X||\le \tau || A||$.
Otherwise output FAILURE and stop.

%------------------------------------------------------------------------------

\end{enumerate}

%------------------------------------------------------------------------------

\end{description}

%------------------------------------------------------------------------------

\end{algorithm}

%------------------------------------------------------------------------------

We have three options for proceeding with any of 
three Theorems 
\ref{thn}--\ref{thnmb1}
and  thus arrive at the three variants of the algorithm.  
Hereafter we refer to them
as {\bf Algorithms 4.1.1, 4.1.2}, and {\bf 4.1.3}.

The algorithms generate $nr$, $(m+n+r)r$, and $(m+n)r$
i.i.d. Gaussian parameters,
respectively, and then perform order of $(m+r)n^2$, $(m+r)(n+r)^2$,
and $(n+r)mn$
flops, respectively.

By choosing SRFT  matrices $U$, $V$
and $W$,
we can decrease the number of random parameters involved to 
$m+r_+$, $m+n+r_+$, and  $m+n+r_+$, respectively, 
for $r_+$ of order $(r+\log(n))\log(r)$,  $(r+\log(m+n))\log(r)$, and 
 $(r+\log(m+n))\log(r)$, respectively,
and then the order of
 the estimated upper bound on the
failure probability  would increase from  
$3/p^p$, for $p=r_+-r$, to the order $1/r$. 

%------------------------------------------------------------------------------

\begin{remark}\label{releft}
One can compute
 nmbs,  matrix bases, and approximate 
matrix bases of the  left
 trailing singular spaces of a matrix $A$
as the  nmbs, matrix bases and approximate 
matrix bases of the
 trailing singular spaces
of the transposed matrix $A^T$ 
or sometimes by simpler means
(see  Remark \ref{relfts}).
\end{remark}

%------------------------------------------------------------------------------

\begin{remark}\label{renmb}
In the case where $m=n$ 
the computations are
simplified and stabilized numerically.
We can reduce to this case the
computation for 
a rectangular matrix $A$, e.g., by observing that  
\begin{itemize}
  \item%1
 $\mathcal N(A)=\mathcal N(A^TA)$,  
\item%2
$\mathcal N(A)=\mathcal N(B^TA)$ if $A,B\in \mathbb R^{m\times n}$
 and if the matrix $B$ has full rank $m\le n$,  
\item%3
 $(A~|~O_{m,m-n}){\bf u}={\bf 0}_{m}$
if and only if $A\widehat {\bf u}={\bf 0}_{m}$ provided that $m\ge n$ and 
$\widehat {\bf u}=(I_n~|~O_{n,m-n}){\bf u}$,
\item%4
$(A^T~|~O_{n,m-n}){\bf v}={\bf 0}_{n}$ if and only if 
$\widehat {\bf v}={\bf 0}_{n}^T$ provided  that $m< n$ and 
$\widehat {\bf v}=
(I_m~|~O_{n-m,m}){\bf v}$.   
\end{itemize}
Furthermore, here is an alternative option. Represent
an $m\times n$ matrix $A$ for $m>n$ 
as a block vector 
$A=(B_1^T~|~B_2^T~|~\dots ~|~B_h^T)^T$
for
  $k_i\times n$ blocks $B_i$, $i=1,\dots,h$, and
$\sum_{i=1}^hk_i= m$. Note that $\mathcal N(A)=\cap_{i=1}^h \mathcal N(B_i)$ 
and apply \cite[Theorem 6.4.1]{GL13} to
 compute the intersection of null spaces. 
\end{remark}

%------------------------------------------------------------------------------

\begin{remark}\label{rerecr}
{\rm Recursive randomized approximation of the bases of singular spaces.}
 Given a matrix $A$ and two small positive values $\eta$ and $\eta'<\eta$, suppose that
we have 
computed the integers $\rho=\rank_{\eta}(A)$   
and $\rho'=\rank_{\eta'}(A)$, by applying Algorithm  \ref{algldbs} or 3.1+, as
well as 
 an approximate basis $Y=Y_{\eta}$ for the trailing singular space $\mathbb T_{A,\rho}$,
by applying Algorithm 3.1t, 3.1t+,  4.1.1, 4.1.2,  or 4.1.3.
Now suppose that we seek   
an approximate matrix basis $Y'=Y_{\eta'}$ for the trailing singular space $\mathbb T_{\rho',A}$.
Then again we can apply one of  these algorithms 
to the matrix $A$, 
but we can apply it to the matrix $AY$ instead,
by increasing
 the precision $u$
of computing
to $u'>u$ 
such that 
$2^{u'}=O(\sigma_{\rho'+1}(A))$, 
but decreasing the arithmetic cost 
 by a factor of $n/\rho$,
which is substantial if $\rho\ll n$.
Correctness of this recipe follows from 
Theorems \ref{thn}--\ref{thtrail},
and the
 approach can be extended recursively.
\end{remark}

%------------------------------------------------------------------------------

\medskip

\medskip

\medskip

\medskip

{\Large \bf \em PART II: Augmentation and Additive Preprocessing}

%------------------------------------------------------------------------------

\section{Analysis of Randomized Augmentation}\label{saug}

%------------------------------------------------------------------------------

Our algorithms of the previous 
two sections rely on the power of randomized augmentation
and additive preprocessing, which we prove in this and the next two sections. 

Row and column permutations make no impact on the singular values of a matrix,
 and so we restrict our next study to western, northern and northwestern augmentation,
that is, to appending Gaussian rows on the top of a matrix or Gaussian columns on the left of it.
Furthermore, western  augmentation for a matrix turns into northern  augmentation
for its transpose and vice versa, and so it is sufficient to analyze 
 western and northwestern augmentation. 

In the next two subsections we  prove the same quite reasonable upper bound 
on the condition numbers of two matrices
obtained from the same ill-conditioned matrix by means of
western and northwestern  augmentation, respectively, but in order to 
yield this upper bound,  the  northwestern 
augmentation requires about twice as many random parameters.
Our tests in Section \ref{sexp} complement these results by 
clearly showing superior performance 
of 
western versus northwestern augmentation
as well as versus additive preprocessing.  
Some potential applications, however, may require 
northwestern rather than western augmentation (see, e.g., the end of Section \ref{srel}).  

 %------------------------------------------------------------------------------

\subsection{Analysis of western and northern augmentation}\label{srwn}

%------------------------------------------------------------------------------
 
%\medskip

{\bf Assumption 1.} 
We will
simplify our presentation by omitting the restriction 
"with probability 1".  For example, by saying that a random matrix $A$ has 
full rank or showing an estimate for the norm $||A^+||$, we will
assume by default (although will not state explicitly)
 that these property or 
estimate hold with
probability 1.

\medskip

%------------------------------------------------------------------------------

\begin{theorem}\label{thwst}  (Cf. Remark \ref{rewcnd} and Definition \ref{defenu}.)
Assume
that an $m\times n$  matrix $A$ 
is normalized and has numerical rank $\rho$.
Define its randomized 
western augmentation 
by the map
%\begin{equation}\label{eqnwaug}
$A\Longrightarrow K=(U~|~A)$ 
%\end{equation} 
for  $U\in \mathcal G^{m\times q}$.
% and $q\ge m-\rho$.
Then 
\begin{equation}\label{eqnrmk}
||K||\le ||A||+||U||= 1+\nu_{m,q},
\end{equation}
for the random variable $\nu_{m,q}$
defined in Section \ref{sldng}
and Appendix \ref{srrm}. 
Furthermore 

(i) the matrix $K$ is 
 rank deficient or ill-conditioned if $q+\rho<l=\min \{m,n\}$.

(ii) Otherwise 
it has full rank 
and 

(iii) satisfies the following bound,

\begin{equation}\label{eqk+} 
||K^+||\le n^+_{m,q,\rho,A}=\max\{1,\nu_{m-\rho,q}^+\}\frac{1+\nu_{\rho,m-\rho}}{\sigma_{\rho}(A)}.
\end{equation}
\end{theorem}

\begin{proof}
 Readily verify (\ref{eqnrmk})
and
%deduce from Theorem \ref{thsignorm}
part (i). 
Deduce part (ii) from Theorem \ref{thdgr}. 
 
It remains to prove bound (\ref{eqk+})
provided that $l\le q+\rho$.

 With no loss of generality, we can replace the matrix  
$A$ by the diagonal matrix $\Sigma_A$ of its singular values,
that is, we can write 
   
$$A=
\Sigma_A=\diag(\Sigma_{\rho},\Sigma'_{m-\rho,n-\rho})~{\rm and}~
K=\begin{pmatrix}
U_0   &  \Sigma_{\rho}  & O_{\rho,n-\rho}  \\
\bar U   &   O_{m-\rho,\rho}  & \Sigma' _{m-\rho,n-\rho}
\end{pmatrix}$$ where $||\Sigma'_{m-\rho,n-\rho})||=\sigma_{\rho+1}(A)$,
$\Sigma_{\rho}=\Sigma_{\rho,A}=\diag(\sigma_j(A))_{j=1}^{\rho}$,
$U_0\in \mathcal G^{\rho\times q}$, and
$\bar U\in \mathcal G^{(m-\rho)\times q}$.
Indeed, we arrive at these equations by applying the orthogonal map 
$K\rightarrow S_A^TK\diag(I_q, T_A)$, which also induces the map
$A\rightarrow S_A^TAT_A=\Sigma_A$. Here $S_A$ and $T_A$
are the matrices  of the singular vectors in SVD $A=S_A\Sigma_AT_A^T$,
and we note that $S_A^TU\diag(I_q, T_A)\in \mathcal G^{m\times q}$
by virtue of  Lemma \ref{lepr3} and that the map 
preserves all singular values of the matrix $K$. 
We call such maps {\em Gaussian diagonalization}. 

Furthermore with no loss of generality we can assume that $n=\rho=l\le m$
and that $$K=
\begin{pmatrix}
U_0   &  \Sigma_{\rho}    \\
\bar U  &   O_{m-\rho,\rho}
\end{pmatrix}\in \mathbb R^{m\times (n+q)}$$
because  the $n-q$ rightmost columns of the matrix $K$ are filled with 
zeros, and we could just delete them.
Note that Theorem  \ref{thdgr} and Assumption 1 together imply that
the $(m-\rho)\times q$ matrix $\bar U$ has full rank, and so
 $\rank(K)=m$ because $q+\rho\ge m$.

Then again apply Gaussian diagonalization by 
writing 
$$\widehat K=\diag(I_{\rho},S_{\bar U}^T)K \diag(T_{\bar U}^T,I_{\rho})=
\begin{pmatrix}
U_{00} & U_{01}   &  \Sigma_{\rho}\\
 \Sigma_{\bar U}'  & O_{m-\rho,q+\rho-m} &   O_{m-\rho,\rho}  
\end{pmatrix}$$ where
$\bar U=S_{\bar U}\Sigma_{\bar U}T_{\bar U}^T$ is SVD, 
$\Sigma_{\bar U}=(\Sigma'_{\bar U}~|~O_{m-\rho,q+\rho-m})$,
$(U_{00} ~|~ U_{01})=U_0 T_{\bar U}^T\in \mathcal G^{\rho\times q}$ 
by virtue of Lemma \ref{lepr3},  and $U_{00}\in \mathcal G^{\rho\times(m-\rho)}$.
Note that $||K^+||=||\widehat K^+||$.

The  $m\times m$
submatrix $$\bar K=\begin{pmatrix}
U_{00}   &  \Sigma_{\rho}\\
 \Sigma_{\bar U}' &   O_{m-\rho,\rho}  
\end{pmatrix}$$ of the matrix $\widehat K$,
 obtained by deleting the submatrix 
$\begin{pmatrix}
 U_{01}   \\
 O_{m-\rho,q+\rho-m}   
\end{pmatrix}$, is nonsingular 
by virtue of Theorem \ref{thdgr} (cf. Assumption 1).
Moreover 
 $||K^+||=||\widehat K^+||\le ||\bar K^{-1}||$
by virtue of Lemma \ref{faccondsub}.

Now observe that
%\begin{equation}\label{eqk-1}
$$\bar K^{-1}=
\begin{pmatrix}
  O_{m-\rho,\rho}  & (\Sigma_{\bar U}')^{-1} \\
  \Sigma_{\rho}^{-1}  & -\Sigma_{\rho}^{-1}U_{00} (\Sigma_{\bar U}')^{-1} 
\end{pmatrix}=\diag(I_{\rho},\Sigma_{\rho}^{-1})\begin{pmatrix} 
  O_{m-\rho,\rho}  & I_{\rho} \\
  I_{m-\rho,m-\rho}  & U_{00}  
\end{pmatrix}\diag(I_{\rho},\Sigma_{\bar U}')^{-1}),$$
%\end{equation}
$$\Big|\Big|\begin{pmatrix} 
  O_{m-\rho,\rho}  & I_{\rho} \\
  I_{m-\rho,m-\rho}  & U_{00}  
\end{pmatrix}\Big|\Big|\le 1+||U_{00}||=1+\nu_{\rho,m-\rho},$$ 
$$||\diag(I_{\rho},(\Sigma_{\bar U}')^{-1})||=\max\{1,||(\Sigma_{\bar U}')^{-1}||\}=
\max\{1,||\bar U^{+}||\}=\max\{1,\nu_{m-\rho,q}^+\},$$
 $$||\diag(I_{\rho},\Sigma_{\rho}^{-1})||=\max\{1,||\Sigma_{\rho}^{-1}||\}=
\frac{1}{\sigma_{\rho}(A)}.$$ 
The latter equation follows because $\sigma_{\rho}(A)\le||A||$ and
because $||A||=1$ by assumption.

Combine the above observations with the bound $||K^+||\le ||\bar K^{-1}||$
 and obtain (\ref{eqk+}).
\end{proof}
Next we combine Theorems \ref{thwst}, \ref{thsignorm}, and \ref{thsiguna}
 and obtain the following  bounds on the
expected values of the norms $||K||$ and $||K^+||$
 (excluding the case where $q+\rho=m$
and the auxiliary random variable $\nu^+_{m-\rho,q}$ of (\ref{eqk+dl})
has no expected value).

\begin{corollary}\label{cocndpr} Under the assumptions of Theorem \ref{thwst}, it holds that
$$\mathbb E(||K||)<2+\sqrt m+\sqrt q,$$ and if $q+\rho>l=\min\{m,n\}$, then
\begin{equation}\label{eqk+e}
\mathbb E(||K^+||)< \frac{2+\sqrt {\rho}+\sqrt {l-\rho}}{\sigma_{\rho}(A)}~
\max\{1,\frac{e\sqrt {(l-\rho)}}{q+\rho-l}\},~{\rm for}~e=2.71828\dots.
\end{equation}
\end{corollary}

%------------------------------------------------------------------------------

\begin{remark}\label{rewcnd}
Theorem \ref{thwst} and the corollary show that
 the western augmentation is likely to output
a well-conditioned matrix $K$,  
particularly if the ratio $\frac{e\sqrt {(l-\rho)}}{q+\rho-l}$ is small.
We can partly control this ratio by choosing the integer parameter $q$. 
If  we  decrease the ratio below 1, 
then it would hold that 
$\mathbb E(||K^+||)<\frac{2+\sqrt {\rho}+\sqrt {l-\rho}}{\sigma_{\rho}(A)}$,
where the value $\sigma_{\rho}(A)$ is not small by assumption.
\end{remark}

By applying Theorem \ref{thwst} and the corollary to the matrix $A^T$, 
we can extend them
to  northern augmentation,
that is, to appending a Gaussian block of $s\ge n-\rho$ rows on the top of the matrix $A$.

%------------------------------------------------------------------------------

\subsection{Analysis of northwestern augmentation}\label{srnw}

%------------------------------------------------------------------------------

\begin{theorem}\label{thnrthwst} 
Assume that an $m\times n$ matrix $A$
is normalized and has numerical rank $\rho$.
Define its 
 randomized northwestern augmentation 
by the map
\begin{equation}\label{eqnwaug} 
A\rightarrow K=
\begin{pmatrix}
W   &   V^T   \\
 U   &  A
\end{pmatrix} 
\end{equation}
% define its  northwestern augmentation
where $W\in \mathcal G^{s\times q}$, 
$U\in \mathcal G^{m\times q}$, $V\in \mathcal G^{n\times q}$,  and
the matrices $U$, $V$, and $W$ are filled with i.i.d. Gaussian variables.
Then 
\begin{equation}\label{eqnrmknw}
||K||\le ||A||+\min\{||U||+||(W~|~V^T)||,||V||+
\Big |\Big | 
\begin{pmatrix}
W      \\
 U   
\end{pmatrix} 
\Big |\Big |= 
1+\min\{\nu_{m,q}+\nu_{s,n+q},\nu_{s,n}+\nu_{m+s,q}\}.
\end{equation}
Furthermore 

(i) the matrix $K$ is 
 rank deficient or ill-conditioned if $q+\rho<m$ and if $s+\rho<n$.

(ii) Otherwise 
it has full rank 
and 

(iii) satisfies the following bounds, 
\begin{equation}\label{eqk+nwm} 
||K^+||\le n^+_{m,q,\rho,A},~
{\rm for}~q+\rho\ge m,~~{\rm and} 
\end{equation}
%\begin{equation}\label{eqlrhowd}
%\end{equation}
\begin{equation}\label{eqk+nwn} 
||K^+||\le n^+_{n,s,\rho,A^T},~{\rm for}~
s+\rho \ge n, 
\end{equation}
where the random variables
$n^+_{m,q,\rho,A}$ and $n^+_{n,s,\rho,A^T}$
are
defined by equation (\ref{eqk+}).
\end{theorem}

\begin{proof}
Readily verify (\ref{eqnrmk})
and
%deduce from Theorem \ref{thsignorm}
part (i). 
Deduce part (ii) from Theorem \ref{thdgr}. 
It remains to prove
bounds (\ref{eqk+nwm}) and (\ref{eqk+nwn}). 

Define western augmentation $$\widehat A\rightarrow K=(\widehat U~|~\widehat A),$$
for $K$ of (\ref{eqnwaug}),
$$~\widehat U=\begin{pmatrix}
W^T  \\
 V  
\end{pmatrix}~{\rm and}~\widehat A=\begin{pmatrix}
 U^T   \\
 A^T
\end{pmatrix}.$$ 
Replace the matrix $A$ by $\widehat A$
and
the integers
$m$, $n$, $l$, $q$, and  $\rho$
by $\widehat m=n+q$, $\widehat n=m$, $\widehat l=\min\{n+q,m\}$,
$\widehat q=s$, and  $\widehat \rho=m$,
respectively. 
Note that $n+q\ge q+\rho\ge m$, and so $\widehat l=m$ and
$\widehat \rho+\widehat q=m+s\ge \widehat l=m$, 
which implies  extension of bound (\ref{eqk+})
to this case.
 Obtain (\ref{eqk+nwm}) because 
$\sigma_{\widehat \rho}(\widehat A)=||\widehat A^+||=n^+_{m,q,\rho,A}$.

Likewise define western augmentation 
$$\bar A\rightarrow  K=(\bar U~|~\bar A),$$
for $K$ of (\ref{eqnwaug}),
$$\widehat U=\begin{pmatrix}
W  \\
 U  
\end{pmatrix} 
~{\rm and}~\bar A=\begin{pmatrix}
 V^T   \\
 A
\end{pmatrix}.$$  
Replace the matrix $A$ by $\bar A$
and
the integers
$m$, $n$, $l$, $q$, and  $\rho$
by $\bar m=m+s$, $\bar n=n$, $\bar l=\min\{n,m+s\}$,
$\bar q=q$, and  $\bar \rho=\min\{\rho+s,n\}$,
respectively. Note that 
$m+s\ge q+\rho\ge n$, and so $\bar l=n$ and
$\bar \rho+\bar s=\bar \rho+q\ge \bar l=n$,
which 
implies  extension of bound (\ref{eqk+})
to this case.
 Obtain (\ref{eqk+nwn}) because 
$\sigma_{\bar \rho}(\bar A)=||\bar A^+||=n^+_{l,s,\rho,A}$.
\end{proof}

%------------------------------------------------------------------------------

The theorem indicates that appending Gaussian rows in addition to Gaussian columns 
as well as appending Gaussian columns in addition to Gaussian rows 
is not likely to increase the norm of the Moore-Penrose generalized inverse of a 
normalized matrix.
 
By combining Theorems \ref{thnrthwst}, \ref{thsignorm}, and \ref{thsiguna},
we obtain the following  bounds.
\begin{corollary}\label{conwdpr} Keep the assumptions of Theorem \ref{thnrthwst}; 
in particular keep the definitions of the integers 
$\widehat l$, $\widehat\rho$, $\bar l$, and $\bar\rho$.
Write $e=2.71828\dots$. Then

$$\mathbb E(||K||)< 3+\sqrt q+\sqrt s+\min \{\sqrt m+\sqrt {n+q},\sqrt n+\sqrt {m+s}\}$$
and $\mathbb E(||K^+||)$ satisfies either bound (\ref{eqk+e})
if $q+\rho>m$ or the same bound but with the integer parameter $s$ replacing $q$ if $s+\rho>n$.
\end{corollary}

%------------------------------------------------------------------------------

\subsection{Analysis of weakly randomized northwestern augmentation}\label{sirnw}

%------------------------------------------------------------------------------

In the next section we analyze randomized  
additive preprocessing 
by linking it to northwestern augmentation
 (\ref{eqnwaug}), for $m=n$ and $r=q=s=n-\rho$, which we  modify by
   choosing $W=I_r$ rather than $W\in \mathcal G^{m\times n}$
and  where we allow the Gaussian matrices $U$ and $V$
to depend on one another and even to share all their entries. We call such 
northwestern  augmentation
{\em weakly randomized}.
Next we extend to it Theorem \ref{thnrthwst}.

\begin{theorem}\label{thbpr1} 
Suppose that an 
 $n\times n$  matrix $A$ is normalized and  has numerical rank $\rho$,
%$A\in \mathbb R^{m\times n}$, $\nrank(A)=\rho$, and $||A||\approx 1$,
 $K$ is the matrix of (\ref{eqnwaug}),
 $W=I_r$, and
$U,V\in \mathcal G^{n\times r}$ where $r=n-\rho$.
Then
$$||K||\le ||A||+||U||+||V||+||W||=2+2\nu_{r,n},$$
 the matrix $K$ is nonsingular,
and 
\begin{equation}\label{eqbrn}
|| K^{-1}||\le 1.5 \bar n,~{\rm for}~
\bar n= 
(1+\nu_{\rho,r})\Big (1+\frac{\nu_{\rho,r}}{\sigma_{\rho}(A)}\Big )
\max \{1,\frac{\nu^+_{r,r}}{\sigma_{\rho}(A)}\}\max \{1,\nu^+_{r,r}\}.
\end{equation}
%or $\rho+q>m$ and $\rho+s>n$.
\end{theorem}

\begin{proof}
We 
only estimate the norm $||K^{-1}||$.
At first let $\nrank(A)=\rank(A)=\rho$
and
then 
reduce our study to the case where
 $A=\Sigma_A=\diag(\Sigma_{\rho,A},O_{r,r})$
by combining
Gaussian diagonalization
$K\rightarrow\diag(I_r, S_A^T)K\diag(I_r, T_A)$
and Lemma \ref{lepr3}.
Write
$$K=\begin{pmatrix}
I_r   &   V_0^T  & V_1^T \\
U_0   &  \Sigma_{\rho}  & O_{\rho,r}  \\
U_1   &   O_{r,\rho}  &  O_{r,r}
\end{pmatrix}$$
where
$ U_0, V_0\in \mathcal G^{\rho\times r}$, 
$U_1,V_1\in \mathcal G^{r\times r}$,
and the matrix $K$ is nonsingular
(cf. Theorem \ref{thdgr} and Assumption 1).
Express the inverse $K^{-1}$ as follows,
$$K^{-1}=
\begin{pmatrix}
O_{r,r}   &   O_{r,\rho}  & U_1^{-1} \\
O_{\rho,r}   &  \Sigma_{\rho}^{-1}  & -\Sigma_{\rho}^{-1}U_0U_1^{-1}  \\
V_1^{-T}   &  -V_1^{-T}V_0^T\Sigma_{\rho}^{-1}    &  V_1^{-T}(I_r-V_0^T\Sigma_{\rho}^{-1}U_0)U_1^{-1}
\end{pmatrix}=$$ 
$$\diag(I_r,\Sigma_{\rho}^{-1},V_1^{-T})\Big (I_{n+r}+\diag(O_{n,n},I_r)-
\diag(O_{r,r},F_{\rho})\Big )\diag(I_n,U_1^{-1})$$
where  $$F_{\rho}=\begin{pmatrix}
  O_{\rho,\rho}   & U_0  \\
  V_0^T\Sigma_{\rho}^{-1}    &  V_0^T\Sigma_{\rho}^{-1}U_0
\end{pmatrix}= \diag(I_{\rho},V_0^T\Sigma_{\rho}^{-1})\diag(O_{\rho,\rho},I_r)\diag(I_{\rho},U_0).$$
 Combine the above expressions and deduce that
 $|| K^{-1}||\le \bar n$, for $\bar n$
of equation (\ref{eqbrn}). 

This is the upper bound of
Theorem \ref{thbpr1} decreased by a factor of 1.5.

By sacrificing this factor, we relax 
the  assumption that  $\nrank(A)=\rank(A)$.

Namely, without this assumption, our previous argument implies that
$$K=\begin{pmatrix}
I_r   &  V_0^T  & V_1^T \\
U_0   &  \Sigma_{\rho}  & O_{\rho,r}  \\
U_1   &   O_{r,\rho}  &  \Sigma'_{r,r}
\end{pmatrix}$$
 where the value
$||\Sigma'_{r,r}||=\sigma_{\rho+1}(A)$
 is small since $\nrank (A)=\rho$.
 Apply Theorem \ref{thpert} for $\theta<1/3$
and obtain that $||K^{-1}||< 1.5 \bar n$.
\end{proof}

%------------------------------------------------------------------------------

\begin{remark}\label{rewknw} 
Our  upper bound on the norm $||K^{+}||$ 
involves the factor 
$(\nu_{r,r}^+)^2$. For larger integers $r$, this makes the bound 
inferior to those of the previous two subsections:
the random variable $\nu_{r,r}^+$.  
 has no expected value; its upper bound in part 2 of Theorem \ref{thsiguna},
 although meaningful, is inferior to the bounds on 
the random variables $\nu_{i,j}$ and $\nu_{r,s}^+$
as long as the
integer $|s-r|$ is not close to 0.
 \end{remark} 

%------------------------------------------------------------------------------

\section{Analysis of Randomized Additive Preprocessing}\label{sapaug}

%------------------------------------------------------------------------------

In this section we analyze randomized additive preprocessing 
%(see policy (ii) of Theorem \ref{thbpr}),
\begin{equation}\label{eqrnap}
A\rightarrow C=A+UV^T~{\rm for}~A,C\in \mathbb R^{n\times n}~{\rm and}~
U,V\in \mathcal G^{n\times r}, 
\end{equation}
where the entries of the matrices $U$ and $V$
may depend on each other, and we even allow $U=V$.

We immediately observe the following properties.

\begin{theorem}\label{thaddpr}
Suppose that $A$, $C$, $U$, and $V$
are four matrices of (\ref{eqrnap}).
Then 
\begin{equation}\label{eqapnrm}
||C||\le ||A||+||U||~||V||\le ||A||+\nu_{n,r}^2, 
\end{equation}
the matrix $C$ is nonsingular 
if and only if $\rank (A)+r\ge n$
 (cf.  Assumption 1 and Theorem \ref{thdgr} or \cite{PQ10}), 
and in this case the matrix $C$ is ill-conditioned if
 $\nrank (A)+r<n$. 
\end{theorem}

Next we estimate the norm $||C^{-1}||$
provided that $\nrank(A)+r\ge n$. 
We do this at first by linking   
 additive preprocessing
to augmentation, then directly.

%------------------------------------------------------------------------------
%------------------------------------------------------------------------------

\subsection{Estimation of the norm $||C^{-1}||$ via a link to augmentation}\label{sauggen}

%------------------------------------------------------------------------------

The following theorem 
links 
augmentation  (\ref{eqnwaug}), for $W=I_r$, to
additive  
preprocessing of (\ref{eqrnap}).
%and later we extend Theorem \ref{thkappa2} to the augmentation as well.

%------------------------------------------------------------------------------

\begin{theorem}\label{th5.2exp}
Suppose that $A\in \mathbb R^{n\times n}$,
$U,V\in \mathbb R^{n\times r}$, 
$K
=\begin{pmatrix}
I_r  &   V^T   \\
U   &   A
\end{pmatrix}$,
 and $C=A+UV^T$. 
Write
$\widehat U=\begin{pmatrix}
O_{r,n}  &   I_r  \\
I_n   &  U 
\end{pmatrix}$,
$\widehat V=\begin{pmatrix}
  O_{n,r} &  I_n    \\
 I_r     &  V^T 
\end{pmatrix}$, 
$\widehat U^{-1}=\begin{pmatrix}
  -U  &  I_n   \\
 I_r  &  O_{r,n} 
\end{pmatrix}$,  
$\widehat V^{-1}=\begin{pmatrix}
-V^T  &   I_r  \\
I_n   &   O_{n,r}
\end{pmatrix}$, and
 $D_{n,r}=\diag(I_n,O_{r,r})$. Then 
\begin{equation}\label{eqk}
K=\widehat U\diag(C,I_r) \widehat V,
~C=D_{n,r}\widehat U^{-1}K\widehat V^{-1}D_{n,r}.
\end{equation}

Furthermore both matrices $C$ and $K$ are singular 
or nonsingular simultaneously.

 They 
are singular if $r+\rank (A)<n$. 

If  the matrices $C$ and $K$ are nonsingular, then 
$$K^{-1}= \widehat V^{-1}\diag(C^{-1},I_r)\widehat U^{-1},~
 C^{-1}=D_{n,r}\widehat V K^{-1}\widehat UD_{n,r},$$ 
$||C^{-1}||\le (1+||U||)(1+||V||)||K^{-1}||$,
and $||K^{-1}||\le (1+||U||)~(1+||V||)~\max\{1,||C^{-1}||\}$.
 \end{theorem}

%------------------------------------------------------------------------------

By combining Theorems \ref{thbpr1} and  \ref{th5.2exp} we extend our results 
for weakly randomized northwestern augmentation of (\ref{eqnwaug})
to randomized 
additive preprocessing.
% as follows.

%------------------------------------------------------------------------------

\begin{corollary}\label{coapaug}
Suppose that $A$, $C$, $U$, and $V$
are  matrices of (\ref{eqrnap}), $||A||=1$, 
and $\nrank(A)+r\ge n$, and so 
the matrix $C$ is nonsingular (with probability 1)
(cf. Theorem \ref{thaddpr}). Define  
 $\bar n$ by (\ref{eqbrn}). Then
$$||C^{-1}||\le 1.5 (1+\nu_{n,r})^2\bar n.$$ 
\end{corollary}

%------------------------------------------------------------------------------

\subsection{Direct estimation of the norm $||C^{-1}||$}\label{spmt2}

%------------------------------------------------------------------------------

 At first we  bound  the 
ratio $\frac{\kappa(C)}{\kappa(A)}$ in the case where 
$\rank(A)+r=n$, then extend the bound to 
 the case where $\nrank(A)+r=n$ and
in the next subsection to 
 the case where $\nrank(A)+r\ge n$. 
% by applying Theorem
%\ref{thgsinv}.
 
\begin{theorem}\label{th5.2} 
Suppose that $A,S,T\in \mathbb R^{n\times n}$ and 
$U,V\in \mathbb R^{n\times r}$
for two positive integers $r$ and $n$, $r\le n$,
$A=S\Sigma T^T$ is SVD of the matrix $A$ (cf. (\ref{eqsvd})),
  $S$ and  $T$
are  square
orthogonal matrices,
%an $m \times n$ 
$\Sigma = \diag(\sigma_j)_{j=1}^n$, 
$\rho=\rank (A)=n-r$,
$\sigma_{\rho}>0$,
 and
the matrix $C=A+UV^T$ 
is nonsingular.
Towards Gaussian diagonalization of the matrix $C$, introduce the matrices
%$\Sigma_A=\diag(\sigma_j)_{j=1}^\rho$
%is the $\rho \times \rho$ diagonal matrix 
%of positive singular values of the matrix $A$. 
%$I_{\widehat r,r}=(I_r,0_{r,\widehat r-r})^T$,
\begin{equation}\label{eqstuv}  
S^TU =     \begin{pmatrix}
                        \bar U    \\
                        U_{r}
                \end{pmatrix},
 ~ T^TV =     \begin{pmatrix}
                        \bar V    \\
                        V_r
                \end{pmatrix},
 ~ R_U =        \begin{pmatrix}
                        I_{\rho}     &    \bar U     \\
                        O_{r,\rho}       &       U_{r}
                \end{pmatrix},
 ~ R_V =        \begin{pmatrix}
                        I_{\rho}     &      \bar V    \\
                        O_{r,\rho}       &       V_r
                \end{pmatrix},
\end{equation}
where $U_r$ and $V_r$ are $r\times r$ matrices. Then 

${\rm (a)}~R_U\Sigma R_V^T=\Sigma$,  $R_U\diag(O_{\rho,\rho},I_r)R_V^T=S^TUV^TT$, 
and so 
\begin{equation}\label{eqcauv} 
C = SR_UDR_V^TT^T,~D=\Sigma+\diag(O_{\rho,\rho},I_{r})=\diag (d_j)_{j=1}^n
\end{equation}
where
$d_j=\sigma_j$ for $j=1,\dots,\rho$, $d_j=1$ for $j=\rho+1,\dots,n$.

%------------------------------------------------------------------------------

%\begin{corollary}\label{co5.222}
%Write $\gamma=\frac{||UV^H||}{||A||}$, $q=||R_U||~||R_V||$ and $p=||R_U^{-1}||~||R_V^{-1}||$.
%Suppose $\sigma_{n-r}(A) \leq 1\leq \sigma_1(A)$.
%Then under the assumptions of Theorem \ref{th5.2} we have
%$\frac{1}{q}\max\{|1-\gamma|,~\frac{1}{p}\}\le\frac{\kappa (C)}{\kappa (A)} \leq p~\min\{1+\gamma,q\}$.
%\end{corollary}

%------------------------------------------------------------------------------

Furthermore 
suppose  that $||A||=1$ 
%$U$ and $V$ are Gaussian matrices, 
and  the  $r\times r$ matrices $U_{r}$ 
and $V_r$ are  nonsingular.
Write 
\begin{equation}\label{eqpuv} 
p=||R_U^{-1}||~||R_V^{-1}||~
{\rm and}~f_r=\max \{1,||U_r^{-1}||\}~\max\{1,||V_r^{-1}||\}.
\end{equation}
 Then 

${\rm (b)}~1\le \frac{\sigma_{\rho}(A)}{\sigma_{n}(C)}\le p$ and
%~S^TU$ and $V^TT$ are Gaussian  matrices,

%${\rm (c)}~1-||U||~||V||\le ||C||\le 1+||U||~||V||$,

${\rm (c)}
%$$\max\{1,||U||,||V||,||U||~||V||\}\le q\le \sqrt{(1+||U||^2)(1+||V||^2)},$$
~ p\le (1+||U||)(1+||V||)f_r$.

\end{theorem}

%------------------------------------------------------------------------------

\begin{proof}
Part (a) is readily verified. 
%Part (b) follows from Lemma \ref{lepr3}.
%Combine the relationships
%$||A||-||U||~||V||\le ||C||\le ||A||+||U||~||V^T||$,  $||A||=1$ and $||V^T||=||V||$.

Let us prove part (b). Combine the equations  
$S^{-1}=S^T$, $T^{-1}=T^T$ and 
(\ref{eqcauv}) 
and obtain
$C^{-1}= TR_V^{-T}D^{-1}R_U^{-1}S^T$.

Apply bound (\ref{eqnorm2}), substitute $||S^T||=||T||=1$,
and obtain
$||C^{-1}||\le ||R_V^{-T}||~||D^{-1}||~||R_U^{-1}||$.

%and $||D^{-1}||\le ||R_V||~||C^{-1}||~||R_U||$.
Substitute equations (\ref{eqpuv}),
 $||D^{-1}||=\frac{1}{\sigma_{\rho}(A)}$
(implied by 
the equations $||A||=1$ and (\ref{eqcauv})),
and $||C^{-1}||=\frac{1}{\sigma_{n}(C)}$
and obtain that 
$\frac{\sigma_{\rho}(A)}{\sigma_{n}(C)}\le p$.

Next deduce from (\ref{eqstuv}) and (\ref{eqcauv}) 
that
%\begin{equation}\label{equv}
$$R_V^{-T} =        \begin{pmatrix}
                        I_{\rho}     &    O_{\rho,r}    \\
                       -V_{r}^{-T}\bar V^T      &       V_{r}^{-T}
                \end{pmatrix},
D^{-1}=\Sigma^{-1}+\diag(O_{\rho,\rho},I_r),
 ~ R_U^{-1}=        \begin{pmatrix}
                        I_{\rho}     &      -\bar U  U_r^{-1}  \\
                        O_{r,\rho}       &       U_r^{-1}
                \end{pmatrix}.$$
%\end{equation}
Substitute these expressions
into  the matrix product
$R_V^{-T}D^{-1}R_U^{-1}$
and obtain that $R_V^{-T}D^{-1}R_U^{-1}= \begin{pmatrix}
                         \Sigma^{-1}    &      X  \\
                       Y     &       Z
                \end{pmatrix}$. Consequently
$\frac{1}{\sigma_{n}(C)}=||C^{-1}||=||R_V^{-T}D^{-1}R_U^{-1}||\ge  ||\Sigma^{-1}||=\frac{1}{\sigma_{n}(A)}$.

This completes the proof of part (b).

(c) Observe that  $R_U^{-1} =        \begin{pmatrix}
                        I_{\rho}     &     -\bar U    \\
                        O       &       I_r
                \end{pmatrix}\begin{pmatrix}
                        I_{\rho}     &     O   \\
                        O       &       U_r^{-1}
                \end{pmatrix}$,
 ~ $R_V^{-1}=        \begin{pmatrix}
                        I_{\rho}     &       -\bar V     \\
                        O       &      I_r
                \end{pmatrix}\begin{pmatrix}
                        I_{\rho}     &     O   \\
                        O       &       V_r^{-1}
                \end{pmatrix}$,
$||\bar U||\le ||U||$ and $||\bar V||\le ||V||$.
Then combine these relationships with (\ref{eqpuv}). 
% with the bounds $\max \{||X||,||Y||\}\le ||(X,Y)||\le \sqrt {||X||^2+||Y||^2}$,
%which hold for all matrices $(X,Y)$.
\end{proof}

%------------------------------------------------------------------------------

\begin{corollary}\label{cocca} 
Suppose that $A\in \mathbb R^{n\times n}$ and 
$U,V\in \mathbb R^{n\times r}$
for two positive integers $n$ and $r$ such that $\rho=\rank(A)=n-r$,
and $C=A+UV^T$. Then 
\begin{equation}\label{eqcca} 
||C^+||\le 
(1+||U||)(1+||V||)~
%f_r~{\rm for}f_r=
\max \{1,||U_r^{-1}||\}~\max\{1,\frac{||V_r^{-1}||}{\sigma_{n}(A)}\}.
\end{equation}
\end{corollary}
\begin{proof}
 Equation (\ref{eqpuv}) and parts (b) and (c) of 
Theorem \ref{th5.2} together imply  (\ref{eqcca}).
\end{proof}

%------------------------------------------------------------------------------

\begin{corollary}\label{cocca1}  
Keep the assumptions of Corollary \ref{cocca},
%well-conditioned rank deficient 
but assume that $\nrank(A)=\rho\le \rank(A)$ and 
$U,V\in \mathcal G^{n\times r}$. Then 

{\rm (i)} the matrix $C$ is nonsingular with probability 1
and
$${\rm (ii)}~||C^{-1}||\le 1.5 (1+\nu_{n,r}^2)\max\{1,\nu_{n,r}^+\}\max\{1,
\frac{\nu^+_{r,r}}{\sigma_{\rho}(A)}\}.~~~~~~~~~~~~~~~~~~~~~~~~~~~~~~~~~~~~~~~~~~~$$ 
\end{corollary}  

\begin{proof}
Part (i) follows from Theorem \ref{thdgr}.

Next note that $U_r$ and $V_r$ are Gaussian matrices 
by virtue of Lemma \ref{lepr3}
because  $U,V\in \mathcal G^{n\times r}$.

At first let
the matrix $A$ be  rank deficient and well-conditioned,  
such that $\nrank(A)=\rank(A)$.

Then 
$||C^+||\le (1+\nu_{n,r}^2)\max\{1,\nu_{n,r}^+\}\max\{1,\frac{\nu^+_{r,r}}{\sigma_{\rho}(A)}\}$
by virtue of Corollary \ref{cocca},
and part (ii) of  Corollary \ref{cocca1} follows from 
 Theorems
\ref{thsignorm}
and  \ref{thsiguna}.
 
Finally, as at the end of our proof of Theorem \ref{thbpr1}, 
apply a small norm perturbation of this matrix and
extend this estimate to the general case where  $\nrank(A)=\rho\le \rank(A)$.
\end{proof}

%------------------------------------------------------------------------------

\begin{remark}\label{readdcnd} 
The upper bound of the corollary on the norm $||C^{-1}||$
is quite reasonable. It is proportional to $\nu^+_{r,r}$,
which makes it superior to the bound 
 of Corollary \ref{coapaug},
proportional to $(\nu^+_{r,r})^2$
(cf. Remark \ref{rewknw}).
Moreover in our extensive tests the norm $||C^{-1}||$
has consistently stayed at 
a substantially lower level, namely at the level of 
 the norms $||K^+||$ of the matrices generated
from the same input matrices $A$  by means of
northwestern augmentation. 
\end{remark}

%------------------------------------------------------------------------------

\subsection{Extension to additive preprocessing of rectangular matrices
 with multipliers of larger sizes}\label{srect}

% - - - - - - - - - - - - - - - - - - - - - - - - - - - - - - - - - - - - -

We have bounded the condition number $\kappa (C)$ of the matrix 
$C=A+UV^T$ in two ways -- by linking additive preprocessing 
to augmentation and directly -- and in both cases,  
under the assumptions that $m=n$
 and $\rho=\nrank (A)>n-r$. Next we remove both of these assumptions,
at the price of increasing our upper bound on the norm $||C^+||$ well
above the square 
of the bounds of Corollaries \ref{coapaug} and \ref{cocca1}.
Such an increase may be due to some technicalities of our proof,
such as application of 
the Sherman---\-Mor\-rison--\-Wood\-bury 
formula\footnote{Hereafter we use the acronym {\em SMW}.} 
(cf. \cite[page 65]{GL13}) 
\begin{equation}\label{eqsmw}
C^{-1}=(\Sigma_{C_-}+\bar U\bar V^T)^{-1}=
\Sigma_{C_-}^{-1}-\Sigma_{C_-}^{-1}\bar U(I_{r-r_-}+\bar V^T\Sigma_{C_-}^{-1}\bar U)^{-1}\bar V^T\Sigma_{C_-}^{-1},
\end{equation} 
and this poses a research challenge of improving 
our estimates. 

%------------------------------------------------------------------------------

\begin{theorem}\label{thcca10}  
Keep the assumptions of Corollary \ref{cocca1}, 
but allow that $\rho=\nrank (A)>n-r$.
 Then
\begin{equation}\label{eqaddpr} 
||C^{-1}||\le \nu_{n,n}^+\nu_{n}^+~{\rm if}~r\ge 2n-\rho
\end{equation} 
where $\nu_{n}$ and $\nu_{n,n}$ are bounded in (\ref{eqsst06})
and in part 2 of Theorem \ref{thsiguna}. 

If  $n-\rho<r< 2n-\rho$,
then
\begin{equation}\label{eqaddpr+}
||C^{-1}||\le (1+\gamma\nu_{n,r+\rho-n}^2||C_-^{-1}||)||C_-^{-1}||,~{\rm for}~
\gamma\le \nu_{r+\rho-n}^+\nu_{r+\rho-n,r+\rho-n}^+||C||,
 \end{equation}
 for an auxiliary matrix  $C_-$ such that the 
%upper 
bounds of Corollaries \ref{coapaug} and \ref{cocca1}
on the norm  $||C^{-1}||$ apply to the norm $||C_-^{-1}||$ as well.
\end{theorem}  

\begin{proof}
In the proof we encounter  matrices that are nonsingular with probability 1, by virtue of 
Theorem \ref{thdgr}. Due to Assumption 1,
we invert them with no further comments.

At first write $r_-=n-\rho$, fix $r_+\ge r_-$,
let $$r=r_++n\ge 2n-\rho,$$
and partition the matrices $U$ and $V$  as follows,
% and $U,V\in \mathcal G^{n\times r}$,
 $$U=(U_+~|~U_n)~{\rm and}~V=(V_+~|~V_n)$$ where 
%\begin{equation}\label{eqbar}
$$U_+,V_+\in \mathcal G^{n\times r_+}~{\rm and}~ 
U_n,V_n\in \mathcal G^{n\times n}.$$ 
%Then the matrices $C$ and $U_n$ are nonsingular with probability 1, by virtue of 
%Theorem \ref{thdgr}, and assume that they are  nonsingular.

Note that  $$C=C_++U_nV_n^T=U_n(U_n^{-1}C_++V_n^T)~{\rm for}~C_+=A+U_+V_+^T.$$ Hence 
$$C^{-1}=(U_n^{-1}C_++V_n^T)^{-1}U_n^{-1}~{\rm and}~
||C^{-1}||\le ||(U_n^{-1}C_++V_n^T)^{-1}||~||U_n^{-1}||.$$
 
Recall that $||(U_n^{-1}C_++V_n^T)^{-1}||=\nu_n^+$ and $||U_n^{-1}||=\nu^+_{n,n}$,
and obtain bound (\ref{eqaddpr}).

\medskip

Next we prove bound (\ref{eqaddpr+}). Assume that 
 $$r_-\le r<r_-+n=2n-\rho,$$ 
and partition the matrices $U$ and $V$ as follows,
 $$U=(U_-~|~\bar U)~{\rm and}~V=(V_-~|~\bar V)$$ where 
%\begin{equation}\label{eqbar}
$$U_-,V_-\in \mathcal G^{n\times r_-}~{\rm and}~ 
\bar U,\bar V\in \mathcal G^{n\times (r-r_-)}.$$ 
%Then the matrices $C_-=A+U_-V_-^T$  and $C=A+UV^T=C_-+\bar U\bar V^T$  are nonsingular with probability 1, %by virtue of 
%Theorem \ref{thdgr}, and assume that they are  nonsingular.

Furthermore, write  $C_-=A+U_- V_-^T$ and $C=C_-+\bar U\bar V^T$ and, by applying
Gaussian diagonalization, reduce our study to the case where $C_-$ 
 is the $n\times n$
diagonal matrix of its singular values,
 $C_-=\Sigma_{C_-}$. 

Represent the matrix $C^{-1}$ by applying the SMW formula 
 (\ref{eqsmw}), write

$$S_{r-r_-}=I_{r-r_-}+\bar V^T\Sigma_{C_-}^{-1}\bar U~{\rm and}~\gamma=||(S_{r-r_-}^{-1}||,$$
recall that
 $||\bar U||=\nu_{n,r-r_-}$ and $||\bar V||=\nu_{n,r-r_-}$, 
 and obtain
$$||C^{-1}||\le
 (1+||\bar U||~\gamma~||\bar V||~||C_-^{-1}||)~||C_-^{-1}||=
(1+\gamma~\nu_{n,r-r_-}^2||C_-^{-1}||)~||C_-^{-1}||$$
where the upper bounds of Corollaries \ref{coapaug} and \ref{cocca1}
hold for $||C^{-1}||$ replaced  by   $||C_-^{-1}||$.

\medskip

It remains to
estimate $\gamma$.
Partition the matrices $\bar U$, $\bar V$,   and $\Sigma_{C_-}$ as follows,
$$ \bar V^T=(\bar V^T_{r-r_-}~|~\bar V^T_{n-r+r_-}),~
\bar U^T=(\bar U^T_{r-r_-}~|~\bar U^T_{n-r+r_-}),~{\rm and}~
\Sigma_{C_-}=\diag(\Sigma_{C_{-,r-r_-}},\Sigma_{C_{-,n-r+r_-}})$$
where 
$$\bar U^T_{k},\bar V^T_{k}\in \mathcal G^{(n-r+r_-)\times k}~{\rm and}~
\Sigma_{C_{-,k}}\in \mathbb R^{k\times k},~{\rm for}~k=r-r_-,n-r+r_-.$$
Write $B=I_{r-r_-}+\bar V^T_{n-r+r_-}\Sigma_{C_{-,n-r+r_-}}^{-1}\bar U_{n-r+r_-}
$ and
$\bar B=\Sigma_{C_{-,r-r_-}}\bar V^{-T}_{r-r_-}B+\bar U_{r-r_-}$
and note that
$$S=I_{r-r_-}+\bar V^T\Sigma_{C_-}^{-1}\bar U=
B+\bar V^T_{r-r_-}\Sigma_{C_{-,r-r_-}}^{-1}\bar U_{r-r_-}=
\bar V^T_{r-r_-}\Sigma_{C_{-,r-r_-}}^{-1}\bar B,$$
and so
 $$S^{-1}=\bar B^{-1}\Sigma_{C_{-,r-r_-}}\bar V^{-T}_{r-r_-}~{\rm and}~ 
\gamma=
||S^{-1}||\le ||\bar B^{-1}||~||\Sigma_{C_{-,r-r_-}}||~||\bar V^{-T}_{r-r_-}||.$$
Note that 
$$||\bar V^{-T}_{r-r_-}||=\nu_{r-r_-,r-r_-}^+,~||\Sigma_{C_{-,r-r_-}}||=||C||,
~{\rm and}~ ||\bar B^{-1}||=\nu_{r-r_-}^+,$$
and so $$\gamma\le \nu_{r-r_-,r-r_-}^+\nu_{r-r_-}^+||C||.$$
This completes our proof of bound (\ref{eqaddpr+}).
\end{proof}

Next we
extend Theorem \ref{thcca10}
to the case where $m\neq n$. With no loss of generality we let
$m\ge n$.

\begin{theorem}\label{th5.3}  
Assume that $A$ is an
$m\times n$ matrix such that $||A||=1$,
$\nrank(A)=\rho\ge n-r$, 
$m\ge n>\rho$, $U\in \mathcal G^{m\times r}$,
$V\in \mathcal G^{n\times r}$, and $C=A+UV^T$.

Then 
\begin{equation}\label{eqapnrmn}
||C||\le ||A||+||U||~||V||\le ||A||+\nu_{m,r}\nu_{n,r}, 
\end{equation}
the matrix $C$ has full rank (with probability 1), and
bounds of Theorem  \ref{thcca10} apply to the norm $||C^+||$
replacing the norm $||C^{-1}||$.
\end{theorem}
\begin{proof}
We only estimate the norm $||C^+||$.

By applying Gaussian diagonalization reduce the problem to the case where the  matrix 
$A$ is replaced by the diagonal matrix $\Sigma_A$ of its singular values.

Pre-multiply the equation $C=A+UV^T$ by the matrix $I_{n,m}=(I_n~|~O_{n,m-n})$,
write $C_n=I_{n,m}C$, $\Sigma_{A,n}=I_{n,m}\Sigma_A$, 
and $U_n=I_{n,m}U$, and obtain that $C_n=\Sigma_{A,n}+U_nV^T$,
$\sigma_j(C)\ge \sigma_j(C_n)$ for all $j$, and so $||C^+||\le||C_n^{-1}||$.
Apply Theorem \ref{thcca10} to the matrices $\Sigma_{A,n}$, 
$U_n$, and $C_n$ replacing the matrices $A$, $U$,  and $C$, respectively.
\end{proof}

%------------------------------------------------------------------------------

\section{Can We Weaken Randomness?}\label{sweak}

%------------------------------------------------------------------------------

\subsection{Structured and sparse randomization: missing formal support for 
its empirical power}\label{sspstremp}

%------------------------------------------------------------------------------

Would the results of the previous two sections
 still hold if
we weaken randomness of
the matrices $U$,  $V$ and $W$
by choosing them sparse, structured, or
defined under other probability distributions
rather than Gaussian?
For the goal of producing
 matrices of full rank 
(with probability 1),
the answer is ``yes"
(cf. \cite[Section 2.13]{BP94}, \cite{PZ15}, 
 \cite{PZa}, and  \cite{PZb}),
but would the pre-processed matrices be also well-conditioned?

The affirmative answer is
known for low-rank approximation by means of random oversampling, 
but only for a narrow 
class of structured preprocessing, and such results have only been proven
at the price of allowing a much greater 
probability of failure versus Gaussian preprocessing. 
These results can be readily extended to augmentation and additive preprocessing
(see the next subsection).

In our tests we have observed consistently that replacing Gaussian  preprocessors 
by sparse and structured preprocessors of a much wider class
neither weakens the efficiency of our preprocessing
nor increases the frequency  of its failure. 
In Sections \ref{sdual}--\ref{sdualnw} we 
formally support these observations,
by applying our techniques
of duality and derandomization.

%------------------------------------------------------------------------------

\subsection{Structured randomization with SRFT 
and subcirculant matrices}\label{sweak1}

%------------------------------------------------------------------------------
 
Preprocessing with SRFT  
structured matrices (cf. Appendix \ref{ssrft})
is efficient for
 low-rank approximation
of a matrix by means of random 
oversampling  (cf. \cite[Section 11]{HMT11}).
By using Theorem \ref{thsrft} and Remark 
 \ref{resrft}, we readily extend this property to the case of randomized 
augmentation and additive preprocessing.
Namely, as in the case of low-rank
approximation, SRFT preprocessing still works
 efficiently for the worst case input 
with a probability close to 1, 
although this is proven
only for SRFT  matrices of larger size
 (due to using the oversampling parameter $\rho_+-\rho$
in Theorem  \ref{thsrft}
and Remark 
 \ref{resrft}) and at the price of accepting a 
greater probability of failure compared 
to the case of Gaussian  preprocessing
(see Remark \ref{refprob}). 
 \cite[Section 4.6]{HMT11} lists a few other 
classes of structured matrices
as alternatives that have power similar to the SRFT matrices. 

In the case of western augmentation  with SRFT, our analysis
 boils down to 
bounding the norms 
$||U_{00}||$ and $||\bar U^{+}||$ where
the matrices  $U_{00}$ and $\bar U$
are the blocks of the matrix
$S_A^TU=\begin{pmatrix}
U_0   \\
U_1   
\end{pmatrix}$, $S_A$ is the orthogonal matrix  
of the left singular vectors of the $m\times n$ input matrix $A$, 
and $U$ is an $n\times q$ SRFT matrix, for  $q$ satisfying
$$4\Big (\sqrt{m-\rho}+\sqrt{8(m-\rho)m}\Big)^2\log ({m-\rho})\le q\le m.$$
It remains 
to analyze randomized western augmentation based on 
Theorem  \ref{thsrft},
which implies that the probability of failure is $O(\frac{1}{m-\rho})$
in our case.
If $m-\rho\gg \log (m)$, then we can obtain a little
more favorable lower estimates for $q$, based on  Remark \ref{resrft}. 

The result is readily extended to the case of northern 
and then northwestern augmentation with SRFT. 
Similarly we can extend our  analysis of additive preprocessing 
based  on  Theorem  \ref{thsrft} and Remark \ref{resrft}.
We omit the details.

Fact \ref{facsrft} implies that
Theorem \ref{thsrft} and Remark 
 \ref{resrft} still hold if
 we replace an $n\times \rho_+$ SRFT matrix
by the matrix $\frac{n}{l_+}~CR$, that is, by the 
scaled product 
of an $n\times n$ random circulant matrix 
$Z=(z_{i-j\mod n})_{i,j=0}^{n-1}$ and an
 $n\times \rho_+$  random matrix $R$ of  Theorem \ref{thsrft}.
If we further  
substitute the matrices $(I_{\rho_+}~|~O_{n,\rho})^T$
or $(|~O_{n,\rho}~|~I_{\rho_+})^T$ for
the factor $R$ of the SRFT, then instead of 
SRFT matrices we arrive at subcirculant matrices (defined in Appendix \ref{scsct}).
If the input matrix $A$ is subcirculant or, more generally, has 
structure of Toeplitz type (cf. \cite{P01} on these matrices), then 
using  subcirculant preprocessing
 is  attractive
because this preserves matrix structure.
We cannot extend the proofs of 
Theorem   \ref{thsrft} and Remark 
 \ref{resrft} from SRFT matrices to such blocks,
but in our extensive tests the  impact of
   our preprocessing    
on the condition numbers of  the input
matrices
remained about the same when
 we properly
scaled these blocks
and used  them
 instead of 
SRFT or Gaussian matrices.
In the next subsections we provide some formal support for 
these empirical observations.

%------------------------------------------------------------------------------

\subsection{Dual additive preprocessing}\label{sdual}

%------------------------------------------------------------------------------

According to our study,  Gaussian and SRFT  augmentation and additive
preprocessing are {\em universal}, that is, produce well-conditioned 
matrices of full rank
with a probability close to 1
{\em for any $m\times n$ input matrix} having numerical rank $\rho<\min\{m,n\}$.

Next we observe (cf. Section \ref{sintro}) 
that additive preprocessing with any well-conditioned matrix of full rank
applied to {\em average input matrix} defined under the Gaussian probability distribution
 is as efficient as  Gaussian  preprocessing.
It follows that preprocessing with a sparse and structured well-conditioned matrix of full rank
 is 
efficient when it is applied to a statistically typical input matrix,
that is, to almost any matrix with a narrow class of exceptions.

Let us specify our duality argument. Assume that we are
given three positive integers $m$, $n$ and $r$, where $m\ge n\ge r$, 
a pair of $n\times r$ 
matrices $U\in \mathbb R^{m\times r}$ and $V\in \mathbb R^{n\times r}$,
and another pair of matrices $\bar U\in \mathbb R^{m\times \rho}$ 
and $\bar V\in \mathbb R^{n\times \rho}$, for $\rho=n-r$.
Then write $A=\bar U\bar V^T$ 
and consider additive preprocessing 
%\begin{equation}\label{eqapavr}
$$A=\bar U\bar V^T\rightarrow C=A+UV^T.$$
%\end{equation}
So far we assumed that $U$ and $V$ were Gaussian matrices,
and the matrices $\bar U$ and $\bar V$ were fixed.
In the {\em dual case} we assume that  $U$ and $V$ is any
pair of well-conditioned matrices of full rank $r$ and 
that the matrices
$\bar U$ and $\bar V$ are Gaussian; 
then we call the matrix $A=\bar U\bar V^T$
{\em factor Gaussian of rank} $\rho$.
Furthermore 
we call a matrix $\tilde A=A+E$ 
a {\em small-norm perturbation of a factor Gaussian matrix of rank} $\rho$
if the norm $||E||$ is small in context.

Clearly,
our analysis in the previous section can be immediately extended to the case
where additive preprocessing is applied to
a small-norm perturbation $\tilde A$ of
 average  factor Gaussian matrix $\tilde U\tilde V^T$ having rank $\rho$,
\begin{equation}\label{eqapavr}
\tilde A=\bar U\bar V^T+E\rightarrow \tilde C=\tilde A+UV^T,
\end{equation}
for matrices $\bar U$, $\bar V^T$, and $E$ specified above and for
 any fixed pair of $n\times r$ well-conditioned normalized 
matrices $U$ and $V$ of full rank $r$.
Here we assume that average matrix $\tilde A$
is defined over all pairs of Gaussian matrices $\bar U$
and $\bar V$, which may 
depend on one another and may even coincide with 
one another.
This result promises {\em sidnificant simplification of additive preprocessing}
by means of
 enforcing desired 
structure and  patterns of sparseness 
  onto the
matrices $U$ and $V$
and should motivate substantial research 
effort in this direction. 

%------------------------------------------------------------------------------

\subsection{Dual western augmentation}\label{sdualw}

%------------------------------------------------------------------------------

Next we extend our duality results to western augmentation
$A\rightarrow K=(U~|~A)$.
 So far we studied the case where $A$ was a fixed $m\times n$ matrix having
numerical rank $\rho$ and $U\in \mathcal G^{m\times q}$, but  
our next theorem (cf. also Remark \ref{rewstdlprt}) enables us to
extend our analysis to the map 
\begin{equation}\label{eqaugdl} 
\tilde A=\bar U\bar V^T+E\rightarrow \tilde K=(U~|~\tilde A)
\end{equation}
where $\tilde A$ is the same
matrix of (\ref{eqapavr}), that is, a small-norm perturbation of a factor Gaussian
matrix, 
and $U$ is any normalized well-conditioned matrix of full rank.

%------------------------------------------------------------------------------

\begin{theorem}\label{thwstdl}
Assume
that an $m\times q$  matrix $U$ 
is normalized and has full numerical rank $l=\min\{m,q\}$.
Define its randomized 
western augmentation 
by the map
%\begin{equation}\label{eqnwaug}
$A\Longrightarrow K=(U~|~A)$ 
%\end{equation} 
for  $A=\bar U\bar V^T$, $\bar U\in \mathcal G^{m\times \rho}$,
and $\bar V\in \mathcal G^{n\times \rho}$.
% and $q\ge m-\rho$.
Then 
\begin{equation}\label{eqnrmkdl}
||K||\le ||A||+||U||= 1+\nu_{m,\rho}\nu_{\rho,n},
\end{equation}
for the random variables $\nu_{m,\rho}$ and $\nu_{\rho,n}$
of Definition \ref{defenu}. 
Furthermore 

(i) the matrix $K$ is 
 rank deficient or ill-conditioned if $q+\rho<m$.

(ii) Otherwise 
it has full rank 
and 

(iii) satisfies the following bounds,

\begin{equation}\label{equdl} 
||K^+||\le ||U^+||~{\rm if}~q\ge m,
\end{equation}
\begin{equation}\label{eqk+dl} 
||K^+||\le ||U^+||\max\{1,\nu_{m-q,\rho}^+\nu_{\rho,n}^+\}(1+\nu_{q,\rho}\nu_{q,n})~{\rm if}~m-\rho\le q<m.
\end{equation}
\end{theorem}

%------------------------------------------------------------------------------

\begin{remark}\label{rewstdlprt}
By using Theorem \ref{thpert} one can
extend Theorem  \ref{thwstdl} (and similarly Theorem  \ref{thnwstdl}) 
to the case where $\tilde A=\bar U\bar V^T+E$
for a perturbation matrix $E$ of small norm replaces matrix $A=\bar U\bar V^T$.
\end{remark}

\begin{proof}
 Readily verify bounds (\ref{eqnrmkdl}) and (\ref{equdl}) 
and
%deduce from Theorem \ref{thsignorm}
part (i). 
Deduce part (ii) from Theorem \ref{thdgr}. 
 
It remains to prove bound (\ref{eqk+dl})
provided that $m\le q+\rho$.

By applying Gaussian diagonalization,
reduce this task to the case where
 the matrix  
$U$ is the diagonal matrix $\Sigma_U$ of its singular values,
that is,  
   
$$
K=\begin{pmatrix}
\Sigma_U  &  G_{q,\rho} G_{\rho,n}  \\
O_{m-q,q}   &   G_{m-q,\rho} G_{\rho,n} 
\end{pmatrix}$$ where $G_{i,j}\in \mathcal G^{i\times j}$,
for $i=q$ and $i=m-q$ and for $j=\rho$ and $j=n$.

Write $F=G_{m-q,\rho} G_{\rho,n}$ and let $F=S_F\Sigma_FT_F^T$
be SVD. 

Here $S_F$ and $\Sigma_F$ are $(m-q)\times (m-q)$
matrices (cf. Theorem \ref{thdgr} and Assumption 1)
because $\rho\ge m-q$ by assumption, and $T_F^T\in \mathbb R^{(m-q)\times n}$.

Define the map 
$\bar K\rightarrow \widehat K=\diag(I_q,S_F^T) \bar K\diag(I_q,T_F)=
\begin{pmatrix}
\Sigma_U  &     G_{q,\rho} G_{\rho,n}  \\
O_{m-q,q}   &   \Sigma_F 
\end{pmatrix}$
and note that the matrix $\widehat K$ is nonsingular 
and that
$$||K^+||\le || \widehat K^{-1}||$$
 because the matrices
$\diag(I_q,S_F^T)$ and $\diag(I_q,T_F)$ 
are orthogonal and because $n\ge \rho\ge m-q$.  

Deduce readily  that 
$$||\widehat  K^{-1}||\le 
(1+||G_{q,\rho} G_{\rho,m-q}||)\max\{1,||\Sigma_U^{-1}||\}\max\{1,||\Sigma_F^{-1}||\}.$$

\noindent Recall that $||U^+||=||\Sigma_U^{-1}||\ge 1$ because $||U||=1$ and that
$||G_{q,\rho} G_{\rho,m-q}||\le \nu_{q,\rho} \nu_{\rho,m-q}$.
Hence 
$$||K^+||\le||\widehat   K^{-1}||\le (1+\nu_{q,\rho} \nu_{\rho,m-q})||U^+||\max\{1,||\Sigma_F^{-1}||\}.$$

Obtain bound (\ref{eqk+dl}) by combining this inequality with the 
estimate  
$||\Sigma_F^{-1}||\le \nu_{m-q,\rho}^+\nu_{m-q,n}^+$.

In order to prove the latter estimate, write SVDs 
$G_{m-q,\rho}=S \Sigma T^T$ and $G_{\rho,n}=\bar S \bar \Sigma \bar T^T$
and observe that $S, \Sigma, \bar S, \bar \Sigma \in \mathbb R^{(m-q)\times (m-q)}$
because $\rho\ge m-q$. 

Write $F_-=\Sigma T^T\bar S \bar \Sigma$ and
let $F_-=S_{F_-}\Sigma_{F_-}T^T_{F_-}$ be SVD.
Then $S_{F_-},\Sigma_{F_-},T^T_{F_-}\in \mathbb R^{(m-q)\times (m-q)}$,
and so $SS_{F_-}$ and $T^T_{F_-} T^T$
 are orthogonal matrices. 

Hence we can write $S_F=SS_{F_-}$, $\Sigma_F=\Sigma_{F_-}$, and $T_F^T=T^T_{F_-} T^T$,
defining SVD $F=S_F\Sigma_FT_F^T$. 

Therefore $\Sigma_F=\Sigma_{F_-}$,
$\Sigma_F^{-1}=\Sigma_{F_-}^{-1}$, and so
 $||\Sigma_F^{-1}||=||\Sigma_{F_-}^{-1}||\le ||\Sigma^+||~||\bar \Sigma^+||$.

Substitute $||\Sigma^+||=||G_{m-q,\rho}^+||=\nu_{m-q,\rho}^+$
and $||\bar \Sigma^+||=||G_{\rho,n}^+||=\nu_{\rho,n}^+$
and obtain the claimed estimate for $\Sigma_F^{-1}$.
This completes the proof of bound (\ref{eqk+dl}) and of the theorem.
\end{proof}

Next we combine Theorems \ref{thwstdl}, \ref{thsignorm}, and \ref{thsiguna}
    and obtain the following upper bounds on the
expected values of the norms $||K||$ and $||K^+||$
(excluding the  case where $q+\rho=m$ and the auxiliary random variable 
$\nu^+_{m-\rho,q}$ 
has no expected value).

\begin{corollary}\label{cocndprdl} Under the assumptions of Theorem \ref{thwstdl}, it holds that
$$\mathbb E(||K||)<1+(1+\sqrt m+\sqrt \rho)(1+\sqrt n
+\sqrt \rho),$$
$$\mathbb E(||K^+||)\le \mathbb E(||U^+||)~{\rm if}~q\ge m,$$
and if $\rho<n$ and $m-\rho<q<m$, then
$$\mathbb E(||K^+||)<(1+(1+\sqrt {\rho}+\sqrt {q})(1+\sqrt {n}+\sqrt {q})~
\max\{1,\frac{(m-q)~\rho~e^2}{(\rho+q-m)(n-\rho)}\},~e=2.71828\dots.$$
\end{corollary} 

%------------------------------------------------------------------------------
 
\begin{remark}\label{rewstdl}
Theorem \ref{thwstdl}  shows that
dual western augmentation is likely to produce
a well-conditioned matrix $K$, particularly where
the matrix $U$ 
is well-conditioned,   
the integers $m$ and $n$ 
are not  large, and 
the ratio $\frac{(m-q)~\rho}{(\rho+q-m)(n-\rho)}$ is small.
We can partly control this ratio by choosing the integer parameter $q$,
and we can readily choose a 
well-conditioned or even orthogonal 
matrix $U$.
\end{remark}

By applying the theorem to the matrix $A^T$, we can extend it
to northern augmentation,
that is, to appending a Gaussian block of $s\ge n-\rho$ rows on the top of the matrix $A$.

%------------------------------------------------------------------------------

\subsection{Dual northwestern augmentation}\label{sdualnw}

%------------------------------------------------------------------------------

Our next subject is northwestern augmentation given by  the map
\begin{equation}\label{eqnwaugdl} 
A\rightarrow K=
\begin{pmatrix}
O_{s,q}   &   V^T   \\
 U   &  A
\end{pmatrix}, 
\end{equation}
which is the map
(\ref{eqnwaugdl})
for $W=O_{s,q}$. 

%------------------------------------------------------------------------------

\begin{theorem}\label{thnwstdl} (Cf. Remarks \ref{rewstdlprt} and \ref{renwstdl}.)
Assume that $U\in \mathbb R^{m\times q}$, $V\in \mathbb R^{n\times s}$, 
$||U||=||V||=1$, the matrices $U$ and $V$ have full rank,
$A=\bar U\bar V^T$, 
$\bar U\in \mathcal G^{m\times \rho}$, 
$\bar V\in \mathcal G^{n\times \rho}$, and $K$ is a matrix of (\ref{eqnwaugdl}).

 (i) Then $||K||\le ||U||+||V||+\nu_{m,\rho}\nu_{\rho,n}$.

(ii) If $q+\rho<m$ and $s+\rho<n$, then the matrix $K$ is rank deficient.
Otherwise it has full rank.

\medskip

(iii) If $q\ge m$ or $s\ge n$, then 
 $||K^+||\le ||U^+||~||V^+||~(1+\nu_{m,\rho}\nu_{\rho,n})$.

\medskip

(iv)  If
$m-\rho\le q\le m$ or $n-\rho\le s\le n$,
then there is an auxiliary nonsingular $(\rho+q+s)\times(\rho+q+s)$  matrix $\bar K$
such that 
\begin{equation}\label{eqk+dlnw}
||K^+||\le ||U^+||~\max\{||V^+||,\nu^+_{m-q,\rho}\nu^+_{\rho,n-s}\}~||\bar K^{-1}||
\end{equation}
and 
\begin{equation}\label{eqbrkdl}
||\bar K^{-1}||\le 1+\nu_{l,\rho}~\max\{\nu_{q,\rho},\nu^+_{m-q,\rho}\nu^+_{\rho,n-s}\nu_{\rho,s}\}~+
\nu_{q,\rho}\nu_{\rho,s}(1+\nu_{\rho,l}^2)\nu^+_{m-q,\rho}\nu^+_{\rho,n-s},
\end{equation}
 for $l=\min\{m-q,n-s\}$.
\end{theorem}
\begin{proof}
We will only prove part (iv).
By applying Gaussian diagonalization, 
reduce the task to the case where the matrices 
$U$ and $V$ are replaced by the diagonal matrices of their 
singular values.
Consequently we arrive at the matrix 
$$
\begin{pmatrix}
O_{s,q}   &   \Sigma_{V^T} & O_{s,n-s}  \\
\Sigma_{U}&G_{q,\rho}G_{\rho,s}&G_{q,\rho}G_{\rho,n-s}  \\
O_{m-q,q}   &  G_{m-q,\rho}G_{\rho,s}&G_{m-q,\rho}G_{\rho,n-s} 
\end{pmatrix} $$
where $G_{i,j}\in \mathcal G^{i\times j}$ for $i=q,\rho,m-q$ and $j=s,\rho,n-s$.

By performing row and column interchange we successively arrive at the matrices
$$\begin{pmatrix}
 \Sigma_{V^T} & O_{s,q}   & O_{s,n-s}  \\
G_{q,\rho}G_{\rho,s}&\Sigma_{U}&G_{q,\rho}G_{\rho,n-s}  \\
G_{m-q,\rho}G_{\rho,s}&O_{m-q,q}   & G_{m-q,\rho}G_{\rho,n-s} 
\end{pmatrix},~
\begin{pmatrix}
 \Sigma_{V^T} & O_{s,n-s}   & O_{s,q}  \\
G_{q,\rho}G_{\rho,s}&G_{q,\rho}G_{\rho,n-s}&\Sigma_{U}  \\
G_{m-q,\rho}G_{\rho,s} & G_{m-q,\rho}G_{\rho,n-s}&O_{m-q,q}   
\end{pmatrix},$$ and
$$\widehat K=\begin{pmatrix}
 \Sigma_{V^T} & O_{s,n-s}   & O_{s,q}  \\
G_{m-q,\rho}G_{\rho,s} & G_{m-q,\rho}G_{\rho,n-s}&O_{m-q,q}   \\
G_{q,\rho}G_{\rho,s}&G_{q,\rho}G_{\rho,n-s}&\Sigma_{U}  
\end{pmatrix}.$$ 
Note  that $\sigma_j(\widehat K)=\sigma_j(K)$, for all $j$, and thus  $||K^+||=||\widehat K^+||$.
  
Write $F= G_{m-q,\rho}G_{\rho,n-s}$ and let $F=S_F\Sigma_FT^T_F$ be SVD.
  
For $l=\min\{m-q,n-s\}$,   
write $\Sigma_{l,F}=\diag(\sigma_j(F))_{j=1}^{l}$  and either
$\Sigma_F=(\Sigma_{l,F}~|~O_{l,n-s-l})$
if $m-q\le n-s$
or  $\Sigma_F^T=(\Sigma_{l,F}~|~O_{l-n+s,n-1}^T)$ 
if $m-q\ge n-s=l$. 

By deleting $n-s-m+q$ columns of the matrices $F$ and $\widehat K$  if $m-q\le n-s$ or
their $m-q-n+s$ rows if $m-q\ge n-s$, we obtain nonsingular matrices 
$\bar F$ and 
$$K'=\begin{pmatrix}
 \Sigma_{V^T} & O_{s,l}   & O_{s,q}  \\
G_{l,\rho}G_{\rho,s} & \bar F &O_{l,q}   \\
G_{q,\rho}G_{\rho,s}&G_{q,\rho}G_{\rho,l}&\Sigma_{U}  
\end{pmatrix},$$
with SVD $\bar F=S_{\bar F}\Sigma_{\bar F}T^T_{\bar F}$,  
 where $S_{\bar F},\Sigma_{\bar F}=\Sigma_F,T^T_{\bar F},\bar F\in \mathbb R^{l\times l}$
and $||K^+||=||\widehat K^+||\le ||(K')^{-1}||$ by virtue of Lemma \ref{faccondsub}.  

Note that 
$$K'=\diag(\Sigma_{V^T},\bar F,I_q)\bar K\diag(I_{l+s},\Sigma_{U}),$$ for

$$\bar K=\begin{pmatrix}
I_s & O_{s,l}   & O_{s,q}  \\
\bar F^{-1}G_{l,\rho}G_{\rho,s} & I_l&O_{l,q}   \\
G_{q,\rho}G_{\rho,s}&G_{q,\rho}G_{\rho,l}&I_q  
\end{pmatrix}=I_{l+q+s}+\begin{pmatrix}
O_{s,s} & O_{s,l}   & O_{s,q}  \\
\bar F^{-1}G_{l,\rho}G_{\rho,s} & O_{l,l}&O_{l,q}   \\
G_{q,\rho}G_{\rho,s}&G_{q,\rho}G_{\rho,l}&O_{q,q}  
\end{pmatrix}.$$ Hence 
$(K')^{-1}=\diag(I_{l+s},\Sigma_{U}^{-1})\bar K^{-1}\diag(\Sigma_{V^T}^{-1},\bar F^{-1},I_q)$
 where 
$$\bar K^{-1}=\begin{pmatrix}
I_s & O_{s,l}   & O_{s,q}  \\
-\bar F^{-1}G_{l,\rho}G_{\rho,s} & I_l &O_{l,q}   \\
H&-G_{q,\rho}G_{\rho,l}&I_q  
\end{pmatrix}=I_{l+q+s}-\begin{pmatrix}
O_{s,s} & O_{s,l}   & O_{s,q}  \\
\bar F^{-1}G_{l,\rho}G_{\rho,s} & O_{l,l}&O_{l,q}   \\
H&G_{q,\rho}G_{\rho,l}&O_{q,q}  
\end{pmatrix}$$ and
$$H=G_{q,\rho}G_{\rho,l}\bar F^{-1}G_{l,\rho}G_{\rho,s}-G_{q,\rho}G_{\rho,s}.$$
Therefore

$$||K^+||\le \max\{1,||\Sigma_{U}^{-1}||\}~||\bar K^{-1}||~\max\{1,||\Sigma_{V^T}^{-1}||,||\bar F^{-1}||\},$$

\noindent and so $$||K^+||\le ||U^+||~||\bar K^{-1}||~\max\{||V^+||,||\bar F^{-1}||\}$$

\noindent because $||\Sigma_{U}^{-1}||=||U^+||\ge 1$ and $||\Sigma_{V^{T}}^{-1}||=||V^+||\ge 1$
since $||U||=||V||=1$.

\medskip

Substitute

$$||\bar K^{-1}||\le 1+||G_{l,\rho}||\max\{||F^{-1}||~||G_{\rho,s}||,||G_{q,\rho}||\}+||H||=
1+\nu_{l,\rho}\max\{||F^{-1}||\nu_{\rho,s},\nu_{q,\rho}\}+||H||,$$

$$||H||\le ||G_{q,\rho}||~||G_{\rho,s}||(1+||G_{\rho,l}||^2~||\bar F^{-1}||)\le
\nu_{q,\rho}\nu_{\rho,s}(1+\nu_{\rho,l}^2~||\bar F^{-1}||),$$

\noindent and

$$||\bar F^{-1}||\le||F^{+}||\le||G_{m-q,\rho}^+||~||G_{\rho,n-s}^+||=\nu^+_{m-q,\rho} \nu^+_{\rho,n-s}.$$

By combining the above bounds obtain part (iv) of the theorem.
\end{proof}

%------------------------------------------------------------------------------
 
Combine Theorems \ref{thnwstdl}, \ref{thsignorm}, and \ref{thsiguna},
exclude the case where $q+\rho=m$ or $s+\rho=n$,
in which the auxiliary random variable $\nu^+_{m-\rho,q}$ or $\nu^+_{n-\rho,s}$
has no expected value, and obtain the following  bounds.
\begin{corollary}\label{conwcndprdl} 
 It holds that
$$\mathbb E(||K||)<2+(1+\sqrt m+\sqrt \rho)(1+\sqrt n+\sqrt \rho)$$
under the assumptions of part (i) of Theorem \ref{thwstdl},
$$\mathbb E(||K^+||)\le \mathbb E(||U^+||)~\mathbb E(||V^+||)(1+
(1+\sqrt {\rho}+\sqrt {m})(1+\sqrt {\rho}+\sqrt {n}))$$
under the assumptions of its part (iii), and
$$\mathbb E(||K^+||)\le\mathbb E(||U^+||)~\max\{\mathbb E(||V^+||,\mathbb E(||F^+||)\}
\mathbb E(||\bar K^{-1}||),$$
under the assumptions of part (iv) of Theorem \ref{thwstdl} provided that
$$\mathbb E(||F^+||)\le\frac{\sqrt{(m-q)\rho}}{|q+\rho-m|~|s+\rho-n|},$$
$$\mathbb E(||\bar K^{-1}||\le 1+(1+\sqrt {l}+\sqrt {\rho})~
\max\{(1+\sqrt {\rho}+\sqrt {q}),(1+\sqrt {\rho}+\sqrt {s})\mathbb E(||F^+||)\}+$$
$$(1+\sqrt {\rho}+\sqrt {q})(1+\sqrt {\rho}+\sqrt {s})(1+(1+\sqrt {\rho}+\sqrt {l})^2)\mathbb E(||F^+||),$$
 $q+\rho>m$ or $s+\rho>n$,
and $(q+\rho-m)(s+\rho-n)\neq 0$.
\end{corollary}

%------------------------------------------------------------------------------
 
\begin{remark}\label{renwstdl}
The upper estimates of Theorem \ref{thnwstdl} and Corollary \ref{conwcndprdl}
are a little greater than  those of Theorem \ref{thwstdl} and Corollary \ref{cocndprdl},
but still  show that
 the dual northwestern augmentation is 
likely to produce
a well-conditioned matrix $K$, particularly where
the matrices $U$ and $V$ (of our choice)
are well-conditioned,  
the integers $m$ and $n$ 
are not  large, and
the ratio $\frac{\sqrt{(m-q)\rho}}{|q+\rho-m|~|s+\rho-n|}$ is small,
which we can partly control by choosing the integer parameters $q$ and $s$.
\end{remark}

%------------------------------------------------------------------------------

\subsection{Some policies of derandomization}\label{sdrnd}

%------------------------------------------------------------------------------

The main advantage of dual augmentation and additive preprocessing 
is a chance for simplifying the computations by means of 
choosing sparse and structured auxiliary matrices $U$ and $V$.
The matrices $U$ and $V$ of Section \ref{sprec} 
can be examples: they are
 extremely sparse, very much structured, orthogonal up to scaling,
 and have supported efficient preprocessing in our extensive tests.
Further examples of simple but empirically highly efficient 
preprocessors can be found in \cite{PZa} and \cite{PZb}.

Here is a caveat, however.
Consider
western  augmentation (\ref{eqaugdl})
with a fixed sparse and structured 
preprocessor  $U$ having full numerical rank $\rho_+$.
Although its application is proven to be efficient for average  $m\times n$
matrix $\tilde A$  having numerical rank $\rho\le \rho_+$, it may fail 
for most or all such matrices $\tilde A$ 
from a selected  input
 class.  Similar problems can occur 
for northwestern augmentation and additive preprocessing.

We are likely to exclude running into such bad inputs
 if we  choose universal preprocessing, e.g., with SRFT matrices.
The user and the algorithm designer, however, should weight this 
benefit 
versus simplification of the computations 
with non-universal sparse and structured  preprocessors.

The following two  sample policies 
keep preprocessing less restricted than universal preprocessing:
they do not exclude but just narrow the chances 
for running into bad inputs.

(i) To any fixed input matrix,  apply 
augmentation or additive preprocessing  successively or
concurrently, for a small number of distinct preprocessors,
 pairs of preprocessors, or policies of
preprocessing, assuming that the user 
accepts the output of even a single
successful application.

(ii) Alternatively  choose
 a preprocessor or a pair of preprocessors
 at random from a fixed class of sparse or structured matrices.
Empirically this approach  consistently produces 
 desired outputs  for a variety of inputs
(see Table \ref{tabprec}).
This should encourage
choosing preprocessors at random  
from the classes of matrices defined 
by a small number of real or complex
random parameters, 
or even just by the signs $\pm$
of some integer parameters, as in the tests 
reported in Table \ref{tabprec}. 

\medskip
 
\medskip

\medskip

{\Large \bf \em PART III: Numerical Tests, Summary, and Extensions}

%-----------------------------------------------------------------------------

\section{Numerical Experiments}\label{sexp}

% - - - - - - - - - - - - - - - - - - - - - - - - - - - - - - - - - - - - -
 
Our numerical experiments 
%with random general, Hankel, Toeplitz and circulant matrices 
have been performed in the Graduate Center of the City University of New York 
on a Dell server with a dual core 1.86 GHz
Xeon processor and 2G memory running Windows Server 2003 R2. The test
of the next subsection have been performed by using
Fortran code compiled with the GNU gfortran compiler within the Cygwin
environment, and  all random numbers 
have been generated with the random\_number
intrinsic Fortran function, assuming  the standard Gaussian 
 probability distribution.
%LZ
The tests 
have been performed with MATLAB,
 using its build-in Gaussian random
 number generating function ''randn()'',
except for the random choice of signs $-$ and $+$
 specified at the end of Section \ref{sprec}.
% or the uniform probability distribution 
%over the range $\{x:-1 \leq x < 1\}$.  
We applied no
iterative refinement in the tests.
Their results are in rather  good accordance with
the  results of our formal analysis.

%------------------------------------------------------------------------------

\subsection{Approximation of the leading and trailing singular spaces,
computation of numerical ranks, and 
low-rank appro\-xi\-ma\-tion of 
a matrix }\label{stails}  

%------------------------------------------------------------------------------

Tables \ref{tabSVD_TAIL}--\ref{tabSVD_TAILnmb}  
show the results of our tests where we approximated 
the bases for  the leading and trailing singular spaces 
$\mathbb T_{\rho,A}$ and
$\mathbb T_{A,\rho}$ 
of an $n\times n$ matrix $A$, respectively.
 The matrix had numerical 
rank $\rho$ and the condition number $\kappa(A)=10^{10}$.

We performed the tests for various pairs of $n$ and $\rho$ and
observed reasonably close approximations, having the error norms in the range
from $10^{-6}$ to  $10^{-9}$. The results were similar for 
Gaussian multipliers and Gaussian subcirculant multipliers. 
The latter multiplier
is a leftmost block of an $n\times n$ circulant matrix that 
contains the entire first column of a circulant matrix 
filled with $n$ i.i.d. standard Gaussian variables (cf. Appendix \ref{scsct}). 
% Unlike the tests in \cite{PQZC}, 

Next we  
 describe the tests
in some detail.

\medskip

{\em GENERATION OF THE INPUTS.}

We generated every  $n\times n$ input matrix $A$ 
for our tests of this subsection as follows 
(cf. \cite[Section 28.3]{H02}).
At first we fixed $n$ nonnegative values
$\sigma_1,\dots,\sigma_n$ and the matrix 
$\Sigma_A=\diag(\sigma_j)_{j=1}^n$, then
generated 
 $n\times n$ random 
orthogonal matrices $S_A$ and $T_A$ 
(as the $Q$-factors of Gaussian matrices),
and finally multiplied the three matrices together
with infinite precision
to output
 the matrix $A=S_A\Sigma_AT^A$. 
%(hereafter referred to as {\em error-free ring operations}). 
We performed all the other  computations of this subsection
 with double 
precision, and also rounded all Gaussian values
to 
double 
precision. 
%We performed two refinement iterations for 
%the computed solution of every linear system of equations
%and matrix inverse.

Our
 $n\times n$
matrices $A$ have numerical rank $\rho=n-r$ and
numerical nullity  $r=n-\rho$  (cf. Appendix \ref{sosvdi})
for
$n=64, 128, 256$,  $\rho=1,8,32$. 
We have chosen 
$\sigma_j=1/j,~{\rm for}~j=1,\dots,\rho,$ and
$\sigma_j=10^{-10},~{\rm for}~j=\rho+1,\dots,n,$
 which implied that $||A||=1$ and $\kappa (A)=10^{10}$.

\medskip

{\em APPROXIMATION OF A BASIS FOR THE TRAILING SINGULAR SPACE DIRECTLY.}

%To approximate a basis $B_{A, r}$ for the trailing singular space $\mathbb T_{A,\rho}$ 
At first we  applied Algorithms 4.1.1--4.1.3 and then computed the matrix
 $B_{A,r}Y_{A,r}$ being a least-squares approximation to the matrix $T_{A,r}$.  
Table \ref{tabSVD_TAIL} 
displays  the data from these tests, 
namely,
the average (mean) values of 
 the error norms  
%${\rm rn}_0=||AB_{A,r}||$,
${\rm rn}=||B_{A,\rho}Y_{A,\rho}-T_{A,\rho}||$ 
and of the standard deviations 
observed in
1000 runs of our tests for every pair of $n$ and $r$.
{\em The tests show superior accuracy of the approximations computed
based on randomized northern  augmentation.}
This is in good accordance with the estimates of Theorems \ref{thwst} and \ref{thsiguna}.
In particular the latter theorem implies that an $m\times n$ Gaussian matrix is likely to become 
better conditioned as  
the value $|m-n|$ increases from 1. 

In our tests the accuracy of the outputs has not varied
 much when we replaced Gaussian  matrices by
Gaussian subcirculant ones (of Appendix \ref{scsct}).

%------------------------------------------------------------------------------

\begin{table}[h] 
  \caption{Error  norms of the approximation of the trailing singular space directly}
\label{tabSVD_TAIL}
  \begin{center}
    \begin{tabular}{| *{8}{c|}}
      \hline
	&		&	\multicolumn{3}{|c}{\bf Gaussian Multipliers}					&	\multicolumn{3}{|c|}{\bf Gaussian Subcirculant Multipliers}					\\ \hline
n	&	r	&	\bf Alg. 3.1.1	&	\bf Alg 3.1.2	&	\bf Alg 3.1.3	&	\bf Alg 3.1.1	&	\bf Alg 3.1.2	&	\bf Alg 3.1.3	\\ \hline
64	&	2	&	7.91e-07	&	7.91e-07	&	2.77e-14	&	1.35e-07	&	1.35e-07	&	3.03e-14	\\ \hline
64	&	4	&	2.46e-07	&	2.46e-07	&	4.18e-14	&	3.26e-07	&	3.26e-07	&	4.76e-14	\\ \hline
64	&	8	&	2.70e-07	&	2.70e-07	&	6.48e-14	&	4.90e-07	&	4.90e-07	&	8.93e-14	\\ \hline
128	&	2	&	4.64e-07	&	4.64e-07	&	6.03e-14	&	8.41e-07	&	8.41e-07	&	6.29e-14	\\ \hline
128	&	4	&	5.33e-07	&	5.33e-07	&	1.27e-13	&	1.01e-06	&	1.01e-06	&	1.12e-13	\\ \hline
128	&	8	&	2.88e-06	&	2.88e-06	&	1.79e-13	&	8.82e-07	&	8.82e-07	&	1.81e-13	\\ \hline
256	&	2	&	2.16e-06	&	2.16e-06	&	7.29e-13	&	1.34e-06	&	1.34e-06	&	6.10e-13	\\ \hline
256	&	4	&	2.07e-06	&	2.07e-06	&	2.97e-13	&	3.38e-06	&	3.38e-06	&	4.60e-13	\\ \hline
256	&	8	&	3.66e-06	&	3.66e-06	&	5.86e-13	&	3.80e-06	&	3.80e-06	&	5.06e-13	\\ \hline

    \end{tabular}
  \end{center}
\end{table}

%\clearpage
%------------------------------------------------------------------------------

%\medskip

{\em APPROXIMATION OF A BASIS FOR THE LEADING SINGULAR SPACE
AND LOW-RANK APPROXIMATION OF A MATRIX.}

We have also performed 
 similar tests for the approximation
of the leading singular spaces $\mathbb T_{\rho,A}$ 
of the same $n\times n$
matrices $A$, which had numerical rank $\rho$, 
 and for the
 approximation of such a matrix $A$ with a matrix of rank $\rho$. 
At first we generated $n\times \rho$ Gaussian
  matrices $U$ 
%(for $\rho=1,8,32$)
and Gaussian subcirculant $n\times \rho$ matrices $\bar U$
(in both cases for $\rho=8$ and  $\rho=32$) and
then 
 successively computed the matrices 
$B_{\rho,A}=A^TU$ 
and $B_{\rho,A}=A^T\bar U$ (in order to obtain approximate matrix bases 
for the leading singular space $\mathbb T_{\rho,A}$),  
$B_{\rho,A}Y_{\rho,A}$  as a least-squares approximation to $T_{\rho,A}$, 
%(the latter five matrices by applying error-free  ring operations),
$Q_{\rho,A}=Q(B_{\rho,A})$, and 
$A-AQ_{\rho,A}(Q_{\rho,A})^T$, which is the error matrix
of the approximation of the matrix $A$ based on 
the approximation of a basis for its leading singular space.
Table \ref{tabSVD_HEAD1} displays the data on
the average  error  norms 
${\rm rn}_1=||B_{\rho,A}Y_{\rho,A}-T_{\rho,A}||$ 
and
${\rm rn}_2=||A-AQ_{\rho,A}(Q_{\rho,A})^T||$ 
obtained in 1000 runs of our tests
for every pair of $n$ and $\rho$.
For our choice of $B_{\rho,A}=A^TU$ 
and $B_{\rho,A}=A^T\bar U$,
the computed error  norms were equally small and  
about as small as in Table  \ref{tabSVD_TAIL}.

%------------------------------------------------------------------------------

\begin{table}[h] 
  \caption{Error  norms of the approximation of
 the leading singular spaces and of
low-rank approximation of a matrix}
\label{tabSVD_HEAD1}
  \begin{center}
    \begin{tabular}{| *{7}{c|}}
      \hline
     \multicolumn{3}{|c}{} & \multicolumn{2}{|c|}{\bf Gaussian Multipliers} & \multicolumn{2}{|c|}{\bf  Subcirculant Multipliers} \\ \hline
$\rho$  & ${\rm rn}_i$ & n & \bf{mean} & \bf{std} & \bf{mean} & \bf{std}\\ \hline
		
%1 & ${\rm rn}_1$ & 64 &   3.58e-09   &   1.37e-08   & &  \\ \hline		
%1 & ${\rm rn}_1$ & 128  &   3.55e-09   &   5.71e-09   & &  \\ \hline		
%1 & ${\rm rn}_1$ & 256 &   5.47e-09   &   8.63e-09  & &  \\ \hline		
%1 & ${\rm rn}_2$ & 64 &   3.86e-09   &   1.36e-08   & &  \\ \hline		
%1 & ${\rm rn}_2$ & 128 &   3.91e-09   &   5.57e-09  & &   \\ \hline		
%1 & ${\rm rn}_2$ & 256 &   5.96e-09   &   8.47e-09  & &   \\ \hline	
8 & ${\rm rn}_1$ & 64 &  4.26e-07   &   8.83e-07  &   1.43e-07   &   9.17e-07 	\\ \hline	
8 & ${\rm rn}_1$ & 128 &   4.30e-08   &   1.45e-07  &  4.87e-07   &   4.39e-06  \\ \hline		
8 & ${\rm rn}_1$ & 256 &   3.40e-08   &   5.11e-08   &   6.65e-08   &   3.12e-07  \\ \hline		
8 & ${\rm rn}_2$ & 64 &   5.77e-09   &   1.06e-08  &   6.37e-08   &   4.11e-07 \\ \hline		
8 & ${\rm rn}_2$ & 128&   1.86e-08   &   5.97e-08   &   1.90e-07   &   1.67e-06 \\ \hline		
8 & ${\rm rn}_2$ & 256 &   1.59e-08   &   2.47e-08   &   2.92e-08   &   1.28e-07 \\ \hline				
32 & ${\rm rn}_1$ & 64 &   1.01e-07   &   3.73e-07   &   4.06e-08   &   6.04e-08 \\ \hline		
32 & ${\rm rn}_1$ & 128 &   1.28e-07   &   6.76e-07   &   2.57e-07   &   8.16e-07 \\ \hline		
32 & ${\rm rn}_1$ & 256 &   1.02e-07   &   1.54e-07    &   1.18e-07   &   2.03e-07 \\ \hline		
32 & ${\rm rn}_2$ & 64&   2.30e-08   &   8.28e-08   &   9.66e-09   &   1.48e-08  \\ \hline		
32 & ${\rm rn}_2$ & 128 &   2.87e-08   &   1.45e-07  &   5.50e-08   &   1.68e-07  \\ \hline		
32 & ${\rm rn}_2$ & 256 &   2.37e-08   &   3.34e-08  &   2.74e-08   &   4.48e-08   \\ \hline	
    \end{tabular}
  \end{center}
\end{table}

%------------------------------------------------------------------------------

%\medskip

{\em EXTENSION FROM THE LEADING TO THE TRAILING SINGULAR SPACES.}
 
Finally we approximated the 
 trailing singular spaces $\mathbb T_{A,\rho}$ for
 the same input matrices $A$ as for Table \ref{tabSVD_TAIL},
where  $\rho=n-r$ and $r=1,2,4$,
but applied Algorithm 3.1t.
%\ref{algbastrvl}.
 At first we
applied Algorithm \ref{algldbs}, which
outputs an approximate matrix basis $B_{\rho,A}$
for the leading singular space 
$\mathbb T_{\rho,A}$.
Then we
applied \cite[Algorithm 4.1]{PQ12}
in  order to compute the matrix $B_{A,\rho}=\nmb(B_{\rho,A})$, 
being an
approximate matrix basis 
for the trailing singular space $\mathbb T_{A,\rho}$.
Table  \ref{tabSVD_TAILnmb} displays
the least-squares error  norms  
${\rm rn}=||B_{A,\rho}Y_{A,\rho}-T_{A,\rho}||$. 
%LZ previously it was ||AQQ^T||
They slightly exceed those of Table
 \ref{tabSVD_TAIL}.

%\clearpage

%------------------------------------------------------------------------------

\begin{table}[h] 
  \caption{Error  norms of approximate bases of 
 the trailing singular spaces computed as the nmbs of the bases for the 
leading singular spaces
%in Table \ref{tabSVD_HEAD1}
}
\label{tabSVD_TAILnmb} 
 \begin{center}
    \begin{tabular}{| c | c | c | c |c|}
      \hline

$r$  & $n$& \bf{mean} & \bf{std}\\\hline

1  	 &64&2.13e-07&6.87e-07\\\hline
1 	 &128	&3.12e-07	&7.20e-07\\\hline
1 	 &256	&9.41e-07	&1.49e-06\\\hline
2 	 &64	&1.74e-07	&3.02e-07\\\hline
2 	 &128	&4.79e-07	&1.12e-06\\\hline
2 	 &256	&1.33e-07	&3.04e-06\\\hline
4 	 &64	&7.49e-07	&3.90e-06\\\hline
4 	 &128	&7.18e-07	&2.63e-06\\\hline
4 	 &256	&3.37e-06	&9.27e-06\\\hline

    \end{tabular}
  \end{center}
\end{table}

% - - - - - - - - - - - - - - - - - - - - - - - - - - - - - - - - - - - - -
%\clearpage 

%------------------------------------------------------------------------------

\subsection{Preconditioning tests}\label{sprecondtests}\label{sprec}

% - - - - - - - - - - - - - - - - - - - - - - - - - - - - - - - - - - - - -

Table \ref{tabprec}  covers our tests for the preconditioning  by  means
of randomized additive preprocessing  and augmentation.
The tests show great power of both additive preprocessing and augmentation,
even though we limited randomization to choosing the signs $+$ and $-$
for the nonzero entries of some very sparse and highly structured matrices
$U$, $V$, and $W$. Namely, both our additive preprocessing and augmentation
consistently decreased the condition numbers of the input matrices from about $10^{16}$
to the values in the range from $10^2$ to $5*10^5$.

\medskip

{\em GENERATION OF THE INPUTS.}

We have tested the input matrices of the following classes.

%------------------------------------------------------------------------------

1n. {\em Nonsymmetric matrices $A$ of type I with numerical nullity $r=n-\nrank (A)$.}
%(See the definition of numerical nullity in Appendix \ref{sosvdi}). 
 $A=S\Sigma_{r}T^T$ are $n\times n$
matrices where $S$ and $T$ are $n\times n$ random orthogonal matrices, that is,   
the factors $Q$ in the QR factorizations of random real matrices;
$\Sigma_{r}=\diag (\sigma_j)_{j=1}^n$ is the diagonal matrix such that $\sigma_{j+1}\leq \sigma_j$ for 
$j=1,\dots,n-1, ~\sigma_1=1$, 
the values $\sigma _2,\dots, \sigma_{n-r-1}$ are randomly sampled in the semi-open
interval $[0.1,1)$, $~\sigma_{n-r}=0.1,~\sigma_j =10^{-16}$ for $j=n-r+1,\dots, n,$
and therefore $\kappa (A)=10^{16}$  \cite[Section 28.3]{H02}. 

%------------------------------------------------------------------------------

1s. {\em Symmetric matrices of type I with numerical nullity $r$.}
The same as in part 1n, but for $S=T$.

%------------------------------------------------------------------------------

The matrices of the six other classes have been constructed in the form of $\frac{A}{||A||}+\beta I$,
with the recipes for defining the matrices $ A$ and scalars $\beta$  specified below.

2n. {\em Nonsymmetric matrices of type II with numerical nullity $r$.}
 $ A=(W~|~WZ)$ where $W$ and $Z$ are random orthogonal matrices of sizes $n\times (n-r)$ and 
$(n-r)\times r$, respectively.

%------------------------------------------------------------------------------

2s. {\em Symmetric matrices of type II with numerical nullity $r$.}
 $ A=WW^T$ where $W$ are random orthogonal matrices of size $n\times (n-r)$.

%------------------------------------------------------------------------------

3n. {\em Nonsymmetric Toeplitz-like matrices with numerical nullity $r$.} 
$ A=c(T~|~TS)$ for random Toeplitz matrices $T$ of size $n\times (n-r)$ and $S$ 
of size $(n-r)\times r$ and for a positive scalar $c$ such that $||A||\approx 1$.

%------------------------------------------------------------------------------
 
3s. {\em Symmetric Toeplitz-like matrices with numerical nullity $r$.} 
$ A=cTT^T$ for random Toeplitz matrices $T$ of size $n\times (n-r)$ and 
a positive scalar $c$ such that $||A||\approx 1$.

%------------------------------------------------------------------------------
 
4n. {\em Nonsymmetric Toeplitz matrices with numerical nullity $1$.} 
$ A=(a_{i,j})_{i,j=1}^{n}$ is a Toeplitz  $n\times n$ matrix. Its entries 
$a_{i,j}=a_{i-j}$ are random for $i-j<n-1$, and so the matrix
$ A_{n-1}=(a_{i,j})_{i,j=1}^{n-1}$
is nonsingular (with probability 1)
and was indeed nonsingular in all our tests. The entry $a_{n,1}$ is selected
to annihilate or nearly annihilate $\det A$, that is, to fulfill 
\begin{equation}\label{eqtz0}
\det A=0~{\rm or} \det A\approx 0,
\end{equation}
in which  case the matrix $A$ is singular or ill-conditioned.

%------------------------------------------------------------------------------

4s. {\em Symmetric Toeplitz matrices with numerical nullity $1$.} 
$ A=(a_{i,j})_{i,j=1}^{n}$ is a Toeplitz  $n\times n$ matrix. Its entries 
$a_{i,j}=a_{i-j}$ are random for $|i-j|<n-1$, while the entry $a_{1,n}=a_{n,1}$ 
was selected to satisfy equation (\ref{eqtz0}), which is
 the quadratic equation in this entry. Occasionally it had no real roots, but
then  we repeatedly generated the
 matrix $A$.

We set $\beta=10^{-16}$ for symmetric matrices $ A$
in the classes 2s, 3s, and 4s, so that $\kappa (A)=10^{16}+1$ in these cases.  
For nonsymmetric matrices $ A$ we defined the scalar $\beta$ by an iterative
process such that $||A||\approx 1$ and $10^{-18}||A||\le \kappa (A)\le 10^{-16}||A||$
\cite[Section 8.2]{PIMR10}.

\medskip

{\em RANDOMIZED PREPROCESSING AND TEST RESULTS.}

Table \ref{tabprec} displays the average values of 
the condition numbers $\kappa (C)$ 
and $\kappa (K)$
%%%Guol
of the matrices 
$C=A+UV^T$ and $K=\begin{pmatrix} W  & V^T  \\
%%%Guol
U & A \end{pmatrix}$
over 1000 tests for the inputs in the above classes, 
$r=1,2,4,8$ and $n=128$. 
Here 
  
$$U=\frac{\bar U}{||\bar U||},
~\bar U^T=(\pm I_r~|~O_{r,r}~|~ \pm I_r~|~ O_{r,r}~|~\dots~|~O_{r,r}~|~\pm I_r~|~O_{r,s}),$$
 $s$ is such that $\bar U \in \mathbb R^{n\times r}$,
$$V=\frac{\bar V}{||\bar V||},~\bar V^T=(2I_r~|~O_{r,r}~|~ 2I_r~|~ O_{r,r}~|~\dots~|~O_{r,r}~|~2I_r~|~O_{r,s})-U^T,$$
$W=\frac{\bar W}{||\bar W||}\in \mathbb R^{r\times r}$,
$\bar W$ are circulant matrices, 
  each defined by its first column,
filled with $\pm 1$, and here as well as in the expression for $\bar U$, all signs
$\pm$ turn into $+$ and $-$
with the same probability 0.5, independently of each other.

In our further tests the condition numbers of the matrices $C=A+10^pUV^T$ 
for $p=-10,-5,5,10$ were steadily growing within a factor $10^{|p|}$
as the value $|p|$ was growing. This  showed the importance of proper scaling 
of the additive preprocessor $UV^T$.

Table \ref{tabprec} also displays the results of the similar tests with 
Gaussian matrices $U$, $V$, and  $W$. The results show 
similar power of Gaussian preprocessors and our 
random sparse and structured preprocessors.

%------------------------------------------------------------------------------

\begin{table}[!ht]
\caption{Preconditioning tests}
\label{tabprec}
\begin{center}
\begin{tabular}{|c|c|c|c|c|c|c|c|}
\hline	Type 	&	 r 	&	$\kappa(C)$, Gaussian 	&	$\kappa(K)$, Gaussian 	&	 $\kappa(C)$, structured 	&  $\kappa(K)$, structured 	\\ \hline
$	1n 	$&	1	&	 1.38e+04 	&	 1.80e+04 	&	 1.80e+04 	&	 2.47e+04 	\\ \hline
$	1n 	$&	2	&	 9.07e+03 	&	 9.66e+03 	&	 8.60e+03 	&	 2.17e+04 	\\ \hline
$	1n 	$&	4	&	 6.91e+04 	&	 7.14e+04 	&	 4.94e+04 	&	 2.15e+05 	\\ \hline
$	1n 	$&	8	&	 2.03e+04 	&	 2.20e+04 	&	 2.81e+04 	&	 1.72e+05 	\\ \hline
$	1s 	$&	1	&	 4.48e+03 	&	 5.76e+03 	&	 3.02e+03 	&	 1.95e+04 	\\ \hline
$	1s 	$&	2	&	 2.32e+04 	&	 1.95e+04 	&	 1.43e+04 	&	 8.19e+04 	\\ \hline
$	1s 	$&	4	&	 2.38e+04 	&	 1.89e+04 	&	 5.67e+03 	&	 7.85e+04 	\\ \hline
$	1s 	$&	8	&	 7.49e+04 	&	 3.32e+04 	&	 1.26e+04 	&	 1.62e+05 	\\ \hline
$	2n 	$&	1	&	 6.75e+03 	&	 7.38e+03 	&	 3.79e+03 	&	 4.27e+03 	\\ \hline
$	2n 	$&	2	&	 1.78e+04 	&	 1.75e+04 	&	 1.74e+04 	&	 3.92e+04 	\\ \hline
$	2n 	$&	4	&	 3.91e+04 	&	 4.44e+04 	&	 1.63e+05 	&	 1.78e+06 	\\ \hline
$	2n 	$&	8	&	 4.57e+04 	&	 3.00e+04 	&	 4.72e+04 	&	 4.56e+05 	\\ \hline
$	2s 	$&	1	&	 1.35e+04 	&	 1.72e+04 	&	 6.17e+03 	&	 1.04e+04 	\\ \hline
$	2s 	$&	2	&	 1.07e+04 	&	 8.81e+03 	&	 8.27e+03 	&	 3.68e+04 	\\ \hline
$	2s 	$&	4	&	 2.01e+04 	&	 1.23e+04 	&	 2.93e+04 	&	 1.74e+05 	\\ \hline
$	2s 	$&	8	&	 2.99e+04 	&	 1.77e+04 	&	 1.65e+04 	&	 2.26e+05 	\\ \hline
$	3n 	$&	1	&	 4.62e+04 	&	 6.49e+04 	&	 1.26e+04 	&	 2.02e+04 	\\ \hline
$	3n 	$&	2	&	 2.68e+06 	&	 2.98e+06 	&	 2.61e+04 	&	 5.96e+04 	\\ \hline
$	3n 	$&	4	&	 4.29e+04 	&	 6.28e+04 	&	 3.75e+05 	&	 1.15e+06 	\\ \hline
$	3n 	$&	8	&	 1.22e+05 	&	 1.79e+05 	&	 1.04e+05 	&	 4.00e+05 	\\ \hline
$	3s 	$&	1	&	 5.34e+05 	&	 7.67e+05 	&	 8.43e+05 	&	 1.32e+06 	\\ \hline
$	3s 	$&	2	&	 2.88e+06 	&	 4.07e+06 	&	 1.52e+06 	&	 3.06e+06 	\\ \hline
$	3s 	$&	4	&	 1.44e+06 	&	 1.99e+06 	&	 3.97e+05 	&	 1.30e+06 	\\ \hline
$	3s 	$&	8	&	 9.63e+05 	&	 1.32e+06 	&	 5.95e+05 	&	 2.88e+06 	\\ \hline
$	4n 	$&	1	&	 4.26e+03 	&	 3.67e+03 	&	 3.51e+03 	&	 3.49e+03 	\\ \hline
$	4n 	$&	2	&	 6.51e+03 	&	 9.84e+03 	&	 7.06e+03 	&	 5.58e+04 	\\ \hline
$	4n 	$&	4	&	 4.22e+03 	&	 1.45e+04 	&	 4.03e+03 	&	 1.78e+05 	\\ \hline
$	4n 	$&	8	&	 4.39e+03 	&	 3.40e+04 	&	 4.72e+03 	&	 3.97e+04 	\\ \hline
$	4s 	$&	1	&	 4.06e+05 	&	 4.14e+05 	&	 2.61e+06 	&	 2.50e+06 	\\ \hline
$	4s 	$&	2	&	 1.34e+06 	&	 3.79e+04 	&	 1.09e+05 	&	 3.24e+04 	\\ \hline
$	4s 	$&	4	&	 1.30e+05 	&	 1.51e+04 	&	 1.49e+04 	&	 4.69e+04 	\\ \hline
$	4s 	$&	8	&	 2.85e+04 	&	 1.17e+04 	&	 1.04e+04 	&	 6.95e+04 	\\ \hline

\end{tabular}
\end{center}
\end{table}

%------------------------------------------------------------------------------

\clearpage

\section{Conclusions}\label{srel}

% - - - - - - - - - - - - - - - - - - - - - - - - - - - - - - - - - - - - -

We studied randomized preprocessing for the acceleration of
computations with singular and ill-conditioned  matrices.
We assumed that
an $m\times n$ input matrix $A-E$ of rank $\rho$
has been represented by its approximation $A$, 
with a small perturbation norm $||E||$, 
so that the matrix $A$ had numerical rank $\rho$.
Then we approximated some bases
for the range 
and the null space 
of the matrix $A-E$, which were
the leading and
trailing singular spaces  $\mathbb T_{\rho,A}$ and $\mathbb T_{A,\rho}$
 of 
the matrix $A$,
respectively, associated with its $\rho$ largest singular values and 
 its remaining singular values, respectively.

The customary numerical algorithms solve these problems
by using pivoting, orthogonalization, or  
SVD, but by extending our earlier study in \cite{PQ10}, \cite{PQ12}, an \cite{PQZC}
we applied randomization
instead of these costly techniques
and obtain accurate solution at a 
significantly lower computational cost.
Our null space algorithms reduce the solution of
  homogeneous rank deficient and ill-conditioned linear systems  
 of equations to the similar tasks
for well-conditioned  linear systems of full rank,
which significantly improves the known algorithms for 
this fundamental computational problem. 

Our work continued the study in a stream of our earlier papers,
which empirically demonstrated the preconditioning power 
of randomized augmentation and additive preprocessing. 
Now  we  supplied 
detailed formal analysis 
which supported these empirical observations.

In particular our study has shown greater efficiency of 
western and northern augmentation (that is, appending a block of
random rows or columns to the given matrix)
versus northwestern augmentation
(that is, appending two blocks of
random rows or columns simultaneously)
and additive preprocessing. This can
properly direct  randomized  preprocessing.

Our formal results have been in good accordance with our 
previous and present numerical tests, which have consistently shown
 that great variety 
of  random sparse and structured preprocessors
(even where randomization was very limited)
 usually  are
 as efficient preconditioners as Gaussian ones.
Similar observations  
have been made by ourselves and by
many other researchers about the power of random sparse and structured multipliers versus 
Gaussian  multipliers in
 their applications to low-rank approximation of a matrix and to GENP. 

For a long while formal support for these empirical observations has been missing, 
but  our novel duality techniques has
provided formal support for these empirical
observations. 

Our results motivate derandomization of our preprocessing 
and bolder application of sparse and structured 
preprocessing for the computational problems studied in this paper
 as well as for some other
important problems of matrix computations.
This promises significant acceleration of the known algorithms.

Promising and in some cases surprising findings of this kind 
have been presented also in \cite{PZa} and \cite{PZb}),
and it is a major challenge to find  new classes of efficient preprocessors 
and new areas where our techniques can 
increase substantially the efficiency of the known algorithms.

%------------------------------------------------------------------------------

In the rest of this section, we outline our novel 
application
 of randomized augmentation and additive preprocessing
to supporting GENP.
The papers \cite{PQZ13},  \cite{PQY15}, and  \cite{PZ15} 
cover alternative randomized multiplicative
support of GENP, its motivation and  history.

Suppose that we are given an $n\times n$ matrix $A$
and we try to apply to it GENP and to
avoid  limitations of multiplicative preprocessing
(cf. \cite{PQZ13},  \cite{PQY15}, and  \cite{PZ15}).
Fix  a positive integer $h<n$
and a pair of $n\times h$ matrices $U$ and $V$
and consider
northwestern augmentation and additive preprocessing 
given by the maps  

\begin{equation}\label{eqaugap}
A\rightarrow K=\begin{pmatrix}
I_h  &  V^T  \\
U  & A
\end{pmatrix}~{\rm and}~C=A-UV^T,
\end{equation}
%(cf. \cite{PQ10},  \cite{PQ12},  \cite[Section 6]{PQZ13},
%and \cite{PQZa}), and 
 respectively.
% from  \cite{PQZa}.
Gaussian augmentation and additive preprocessing
 generate $2hn$ Gaussian parameters each;
additive preprocessing requires in addition 
$(2h-1)n^2$ flops. By choosing structured
(e.g., Toeplitz)  matrices $U$ 
and $V$, we can decrease these bounds to
$O(n)$ random parameters and $O(n\log(n))$
flops.

\begin{theorem}\label{thaugap}
Let  $h$ and $n$ be two positive integers.
Let 
 $A$ be an $n\times n$ matrix normalized so that $||A||\approx 1$
and let   $\eta$ denote the maximum numerical nullity 
of its leading square blocks.
Let 
 $U$ and $V$ be the pair of $n\times h$ Gaussian matrices
such that either $U=V$ or 
these two matrices $U$ and $V$ are independent of one another. 
Suppose that equation (\ref{eqaugap})
defines northwestern 
augmentation  and additive preprocessing 
of the matrix $A$,
producing the matrices $K$ and $C$.

(i) Then these  
 matrices  are  
nonsingular with probability 1, and 
their condition numbers can be estimated from above
according to the probabilistic estimates of Sections \ref{saug} and \ref{sapaug}.   

(ii) One can apply the probabilistic estimates of Section \ref{sweak} instead
if  $U$ and $V$ are SRFT matrices, if 
we  choose $h\ge q=cn$, for a 
sufficiently large constant $c$,
 and if 
$$4\Big(\sqrt {\eta}+\sqrt {8\log_2(\eta n)}\Big)^2\log_2(\eta)\le h.$$
\end{theorem}

The claimed results are readily verified for augmentation with Gaussian 
and SRFT matrices producing matrices $K$. We extend them to 
matrices $C$, produced with additive preprocessing,
by applying GENP to the matrix $K$. Indeed  
we arrive at the same task for the matrix $C$ in $h$ elimination steps.
 
In part (i) of Theorem \ref{thaugap}
we can set $h=\eta$ if we know the bound $\eta$,
but otherwise we can try to guess such a bound {\em by actions}.
Namely, assume at first that $\eta\le 1$, set 
$h=1$, apply GENP to the matrix $K$or $C$, and
in the case of failure, increase (e.g., double) $h$ recursively.
 
Let us motivate this policy. 
Define the $\eta$-{\em family} of matrices as the set of all
matrices with the maximal numerical nullity at least $\eta$
for its leading square blocks.
Then already the 1-family
makes up a small fraction of all matrices, and 
 the size of the $\eta$-family 
 decreases very fast as $\eta$ grows.
 
This randomized preprocessing is universal 
and allows us to use  SRFT structure, but 
supports the application of GENP to
the matrices $K$ and  $C$,  
rather than to the original matrix $A$.
Our next goal is 
 the inversion of the  matrix $A$
or the solution of a linear system $A{\bf x}={\bf b}$
simplified by using the output of the above applications.

%or the solution of a linear system $A{\bf x}={\bf b}$

A potential tool is the 
 SMW formula (\ref{eqsmw}), which we can extend by 
expressing the inverse $A^{-1}$ through 
the inverse $K^{-1}$  rather than  $C^{-1}$.

If the assumptions of Theorem \ref{thaugap}
have been satisfied, then the matrix $C$ is likely to be well-conditioned,
but using the SMW formula may still cause numerical problems
at the stages of computing and inverting the matrix
$I_h+V^TC^{-1}U$. 

For a natural antidote, we can perform 
the computations  at these stages with  extended precision.
They involve $O(hn^2)$ flops, versus the order of $n^3$ flops involved
at the other stages and performed with double precision.
This can be attractive when $h\ll n$.

For a large class of well-conditioned matrices $A$,
we can try to avoid numerical problems 
by scaling the matrices $U$ and $V$. 
This is a research challenge, and next we outline some recipes and obstacles.

If the ratio
$\frac{||A||}{||UV^T||}$ is sufficiently large, 
then
$||VC^{-1}U||\le \theta<1$ for a 
 constant $\theta$ not close to 1,
and the diagonally dominant
 matrix $I_h+V^TC^{-1}U$  
  can be computed and inverted with no numerical problems. 
The power of that recipe is limited, however,
because our randomized preprocessing
does not work if the ratio $\frac{||A||}{||UV^T||}$ is too large.

Application of the homotopy continuation techniques 
(cf. \cite[Section 6.9]{P01}, \cite{PKRK06}, \cite{P10})
may help to extend the power of this recipe.

For two other policies pointed out below,
we must also scale  
the matrices $U$ and $V$
in order to have a sufficiently large
ratio $\frac{||A||}{||UV^T||}$,
and then again this 
scaling 
can be in conflict with obtaining
our randomized  support for GENP for the  
matrices
$K$,  $K'$, and/or  $C$.

(i) If we achieve scaling such that $||I-C^{-1}A||\le \theta <1$ for a  
 constant $\theta$ not close to 1,
then Newton's iteration $X_{i+1}=2X_i-X_iAX_i$, $i=0,1,\dots$, initialized
at $X_0=C^{-1}$, converges quadratically right from the start 
to the inverse $A^{-1}$ (cf. \cite[Chapter 6]{P01}).

(ii) Suppose that we seek the solution of a linear system $A{\bf x}={\bf b}$
and that GENP, applied to the matrix $C=A+UV^T$, has output its LU factorization 
being close 
 to the LU factorization of 
the matrix $A$. Then we can
solve the linear system  $A{\bf x}={\bf b}$ accurately
by applying iterative refinement.

\bigskip

\bigskip
\bigskip
\bigskip
\bigskip

\bigskip
\bigskip
\bigskip

%------------------------------------------------------------------------------

{\bf {\LARGE {Appendix}}}
\appendix 
%\bigskip

%------------------------------------------------------------------------------s
\section{Some Basic Definitions and Properties of Matrix Computations}\label{sosvdi}

%------------------------------------------------------------------------------

A real matrix $Q$ is   
orthogonal if $Q^TQ=I$ 
 or $QQ^T=I$.

%------------------------------------------------------------------------------

$||M||_F$ is the Frobenius norm of a matrix $M$.

%------------------------------------------------------------------------------

$A^+=T_A\diag(\widehat \Sigma_A^{-1},O_{n-\rho,m-\rho})S_A^T$ is the 
Moore--Penrose 
pseudo-inverse of the matrix $A$ of (\ref{eqsvd}). 

$\kappa (A)=\frac{\sigma_1(A)}{\sigma_{\rho}(A)}=||A||~||A^+||$ is the condition 
number of an $m\times n$ matrix $A$ of rank $\rho$. Such matrix is 
{\em ill-conditioned} 
if the ratio $\frac{\sigma_1(A)}{\sigma_{\rho}(A)}=||A||~||A^+||$ is large
and otherwise is {\em well-conditioned}. 
%See  \cite{D83}, \cite[Sections 2.3.2, 2.3.3, 3.5.4, 12.5]{GL96}, 
%\cite[Chapter 15]{H02}, \cite{KL94}, \cite{KW92}, and \cite[Section 5.3]{S98}
%on the estimation of matrix norms and condition numbers. 

%A matrix $A$  is  {\em normalized}  if $||A||=1$.

%------------------------------------------------------------------------------

The {\em numerical rank} of an $m\times n$ matrix $A$, denoted $\nrank(A)$,
is the minimal rank of its nearby matrices, and
 $\nnul(A)=n-\nrank(A)$ is
the numerical nullity of $A$.

% and let the matrices $A$ and $B^T$ have the same number of columns. Then
\medskip

Recall the following basic properties. 

\begin{equation}\label{eqnorm2}
||A^T||=||A||\le ||A||_F=||A^T||_F\le {\sqrt n} ||A||,~~||AB||\le ||A||~||B||,~~||AB||_F\le ||A||_F~||B||_F, 
\end{equation}

\begin{equation}\label{eqnormdiag}
||\diag(M_j)_j||=\max_j ||M_j||~{\rm for~any~set~of~matrices}~M_j.
\end{equation}

%------------------------------------------------------------------------------
%See more on matrix norms and SVD in Appendix \ref{sosvdi}.
%------------------------------------------------------------------------------
%\subsection{Generalized inverses, 
%and perturbation bounds}\label{sigipb}
%$A^+$ is a left  inverse of a matrix $A$ of full rank for $m\ge n$ and its right inverse for $m\le n$. 
\begin{equation}\label{eqnrm+}
||A^+||=\frac{1}{\sigma{_\rho}(A)}.
\end{equation}

\begin{lemma}\label{lepr1} 
Suppose $\Sigma=\diag(\sigma_i)_{i=1}^{n}$, $\sigma_1\ge \sigma_2\ge \cdots \ge \sigma_n$,
$F\in \mathbb R^{r\times n}$, and  $H\in \mathbb R^{n\times r}$.
Then 
\begin{itemize}
\item %1
  $\sigma_{j}(F\Sigma)\ge\sigma_{j}(F)\sigma_n$,
$\sigma_{j}(\Sigma H)\ge\sigma_{j}(H)\sigma_n$ for all $j$. 
\item %2
 If also $\sigma_n>0$, then 
 $\rank (F\Sigma)=\rank (F)$ and $\rank (\Sigma H)=\rank (H)$.
\end{itemize}
\end{lemma}

%------------------------------------------------------------------------------

\begin{lemma}\label{lepr2}
$\sigma_{j}(SM)=\sigma_j(MT)=\sigma_j(M)$ for all $j$ if $S$ and $T$ are square orthogonal matrices.
\end{lemma}

%------------------------------------------------------------------------------
%$\sigma_n=\min_{||{\bf x}||=1}||A{\bf x}||$ if $m\ge n$,
% $\sigma_m=\min_{||{\bf x}||=1}||A^T{\bf x}||$ if $m\le n$,
%--------------------------------------------------------------------------
 
\begin{lemma}\label{faccondsub}
For a matrix $A$, its submatrix $A_0$, and a subscript $j$, it holds that  
$\sigma_{j} (A)\ge \sigma_{j} (A_0)$.
%and therefore $\kappa (A)\ge \kappa (A)_0$.
%c) there is a $p\times q$ block submatrix $\widehat A$ of the matrix $A$ 
%such that $\lceil m/p\rceil^{1/2}\lceil n/q\rceil^{1/2}||\widehat A||\ge ||A||$.
\end{lemma} 

%--------------------------------------------------------------------------
 
\begin{theorem}\label{they} 
We have $|\sigma_{j} (C)-\sigma_{j} (C+E)|\le ||E||$
for all $m\times n$ matrices $C$ and $E$ 
and  all $j$.
\end{theorem}
\begin{proof}
See \cite[Corollary 8.6.2]{GL13} or \cite[Corollary 4.3.2]{S98}. 
\end{proof}

%\begin{theorem}\label{thsb} 
%Let 
%$A_n=({\bf a}_1,\dots,{\bf a}_n)\in \mathbb R^{m\times n}$ and 
%$A_r=({\bf a}_{n-r+1},\dots,{\bf a}_n)\in \mathbb R^{m\times r}$
%for $m\ge n$ and $r=1,\dots,n$. Then $\sigma_r(A_r)\ge \sigma_(r+i)(A_{r+i})$
%for $r=1,\dots,n-1$ and $i=1,\dots,n-r$.
%\end{theorem}
%\begin{proof}
%Apply \cite[Corollary 8.6.3]{GL96} to the matrix $AJ_n$ for the reversion matrix $J_n$. 
%\end{proof}

%------------------------------------------------------------------------------

%\section{Some
%Perturbation Bounds}\label{sigipb}

%------------------------------------------------------------------------------
%------------------------------------------------------------------------------

\begin{theorem}\label{thpert} 
Suppose $C$ and $C+E$ are two nonsingular matrices of the same size
and $||C^{-1}E||  =\theta<1$. Then
\begin{itemize}
\item %1
$||I-(C+E)^{-1}C||  \le \frac{\theta}{1-\theta}$ and
$\||(C+E)^{-1}-C^{-1}||\le \frac{\theta}{1-\theta}||C^{-1}||$.
\item %2
In particular, $\||(C+E)^{-1}-C^{-1}||\le 0.5||C^{-1}||$
if $\theta\le 1/3$.
\end{itemize}
\end{theorem}
\begin{proof}
See \cite[Corollary 1.4.19]{S98} for $P= -C^{-1}E$.
\end{proof}

%------------------------------------------------------------------------------

\begin{theorem}\label{thpertq}  \cite[Theorem 5.1]{S95}.
%For two matrices $A$ and $E$ in $\mathbb R^{
%Let  $Q(A)$ and $Q(A+E)$ denote the Q-factors
Assume a pair of
 $m\times n$ matrices $A$ and $A+E$, 
%respectively,
and let the norm $||E||$ be small. Then 
$||Q(A+E)-Q(A)||_F\le \sqrt 2 ||A^+||~||E||_F+O(||E||_F^2$.
\end{theorem}

%------------------------------------------------------------------------------

\section{A Gaussian 
% general, Toeplitz and circulant 
Matrix. Estimates for Its Rank, Norm and Condition Number}\label{srrm}

%------------------------------------------------------------------------------

\begin{definition}\label{defrndm}  
A matrix is said to be {\em standard Gaussian random} (hereafter referred to just as
{\em Gaussian}) if it is filled with i.i.d.
Gaussian random
variables  having mean $0$ and variance $1$. 
$\mathcal G^{m\times n}$ denotes  the class of $m\times n$ Gaussian matrices.
\end{definition}

%------------------------------------------------------------------------------

\begin{lemma}\label{lepr3} {\rm Invariance of the products of Gaussian matrices  
under orthogonal
multiplications.}  \\
%\cite[Proposition 2.2]{SST06}.
Suppose that
 $H\in \mathbb R^{m\times n}$,
$S\in \mathbb R^{k\times m}$, and $T\in \mathbb R^{n\times k}$
for some positive integers $k$, $m$, and $n$,
and suppose that the  matrices $S$ and $T$ are orthogonal.
Then 

(i) $SH\in \mathcal G^{k\times n}$ and $HT\in \mathcal G^{m\times k}$
if $H\in \mathcal G^{m\times n}$ and 

(ii) $SH\in \mathcal G_{k,n}$ and $HT\in \mathcal G_{m,k}$
if $H\in \mathcal G_{m,n}$.
\end{lemma}

%------------------------------------------------------------------------------

Hereafter we call a vector ${\bf t}$ unit if $||{\bf t}||=1$.

%------------------------------------------------------------------------------

\begin{lemma}\label{leprob2} Cf. \cite[Lemma A.2]{SST06}.
Assume two positive integers $n$ and $r$, a real $\mu$, a positive $x$,
a  unit vector ${\bf u}\in \mathbb R^{k\times 1}$, and 
two independent Gaussian vectors 
${\bf g}_k\in \mathcal G^{k\times 1}$ for $k=r$ and $k=r$. 
Then
{\rm Probability}$\{|{\bf u}^T{\bf g}_n-\mu|\le x\}\le \sqrt{\frac{2}{\pi}}x$.
\end{lemma}

%------------------------------------------------------------------------------

\begin{theorem}\label{thdgr}
A Gaussian matrix  has full rank with probability 1.
\end{theorem}
\begin{proof}
At first recall that an algebraic variety of a dimension $d\le N$
in the space $\mathbb  R^{N}$ is defined by $N-d$
polynomial equations and cannot be defined by fewer equations.
(Fact \ref{far1} specifies the dimension of the algebraic variety of 
$m\times n$ matrices of rank $\rho$.)
Now assume a rank deficient $m\times n$ matrix where $m\ge n$, say.
Then the determinants 
of all its $n\times n$ submatrices vanish. This implies 
$\begin{pmatrix} m \\n \end{pmatrix}$
polynomial
equations on the entries, that is, rank deficient matrices form
an algebraic variety of a lower dimension in the linear space 
$\mathbb R^{m\times n}$.
Clearly, such a variety has Lebesgue (uniform) and
Gaussian measures 0, both being absolutely continuous
with respect to one another.
\end{proof}
%\begin{definition}\label{defcdfm}
%Hereafter we simplify the notation by writing $F_{M}(y)=F_{||M||}(y)$ for a matrix $M$.
% and an integer $l=\min\{m,n\}$. 
%\end{definition}

\begin{theorem}\label{thsignorm} See \cite[Theorem II.7]{DS01}
and our Definition \ref{defenu}.

Suppose  that
%$G$ is a  Gaussian $m\times n$ matrix,
$h=\max\{m,n\}$, $t\ge 0$.
 Then
%------------------------------------------------------------------------------

%\begin{enumerate}
(i)   {\rm Probability}$\{\nu_{m,n}>t+\sqrt m+\sqrt n\}\le
\exp(-t^2/2)$, and so

(ii) $\mathbb E(\nu_{m,n})< 1+\sqrt m+\sqrt n$. 
%\end{enumerate}
\end{theorem}
%------------------------------------------------------------------------------

\begin{theorem}\label{thsiguna} 
Suppose that
$m\ge n$ and $x>0$ and let 
$\Gamma(x)=
\int_0^{\infty}\exp(-t)t^{x-1}dt$ 
and 
$\zeta(t)=
%\frac{\sqrt{2m}}{\Gamma(m/2)}(t\sqrt{m/2})^{m-1}\exp(-mt^2/2)=
t^{m-1}m^{m/2}2^{(2-m)/2}\exp(-mt^2/2)/\Gamma(m/2)$
denote the Gamma function. 
Then 
\begin{enumerate}
\item %1 
 {\rm Probability} $\{\nu_{m,n}^+\ge m/x^2\}<\frac{x^{m-n+1}}{\Gamma(m-n+2)}$
for $n\ge 2$,
\item %2
  {\rm Probability} $\{\nu_{n,n}^+\ge x\}\le 2.35 {\sqrt n}/x$ 
for $n\ge 2$ (cf. Remark \ref{resst}), 
\item %3
 {\rm Probability} $\{\nu_{m,1}^+\ge x\}\le
% Probability $\{||A||\le 1/x\} =\int_{0}^{1/x} \zeta(t) dt<
(m/2)^{(m-2)/2}/(\Gamma(m/2)x^m)$ for $m\ge 2$, and
\item %4 
 $\mathbb E((\nu^+_{F,m,n})^2)=m/|m-n-1|$
provided that $n>1$ and $m-n>1$,
while

$\mathbb E(\nu^+_{m,n})\le e\sqrt{l}/|m-n|$ for $e=2.71828\dots$,
$l=\min\{m,n\}$, and $m\neq n$.
\end{enumerate}
%(iii) Probability $\{||(G+A)^+||\ge 2.35x\sqrt {n}\}\le 1/x$
%for any $m\times n$ matrix $A$.
\end{theorem}
\begin{proof}
 See \cite[Proof of Lemma 4.1]{CD05} for part 1
and \cite[Proposition 10.2]{HMT11} for part 4.
Part 2 follows from 
(\ref{eqsst06}) for $A=O_{n,n}$.

Let us deduce part 3. $G\in \mathbb R^{m\times 1}$ is a vector of length $m$.
 So, with probability 1 it holds that
$G\neq 0$, 
$\rank (G)=1$, $||G^+||=1/||G||$.
 Consequently,
$$ {\rm Probability}\{||G^+||\ge x\}=
 {\rm Probability}\{||G||\le 1/x\}\le\int_{0}^{1/x} \zeta(t) dt~{\rm for}~x>0.$$ 

\noindent Note that $\exp(-mt^2/2)\le 1$. Hence
$\int_{0}^{1/x} \zeta(t) dt< c_m\int_{0}^{1/x}t^{m-1}dt=c_m/(mx^m)$ 
where
$\zeta(t)=
%\frac{\sqrt{2m}}{\Gamma(m/2)}(t\sqrt{m/2})^{m-1}\exp(-mt^2/2)=
t^{m-1}m^{m/2}2^{(2-m)/2}\exp(-mt^2/2)/\Gamma(m/2)$
%denote
is the Zeta function and $c_m=m^{m/2}2^{(2-m)/2}/\Gamma(m/2)$.
%(iv) For $m=n$ 
%this is a special case of 
%apply \cite[Theorem 3.3]{SST06}.
 % For extension to any pair $\{m,n\}$
%apply Fact \ref{faccondsub}.
\end{proof}

%------------------------------------------------------------------------------

\begin{remark}\label{resst}
Part 2 of Theorem \ref{thsiguna} provides some
bound on the random variable $\nu_{n,n}^+$, although this bound  
is weaker than the bounds in other parts of the theorem, and
the random variable $\nu_{n,n}^+$
has no expected value. 
\end{remark}

Theorems \ref{thsignorm} and \ref{thsiguna} 
together imply that the expected value of 
the condition number of an $m\times n$ Gaussian matrix 
decreases 
 quite fast as the integer  $|m-n|$
increases from 1. This implies  greater efficiency
of western  and northern  augmentation versus
northwestern one.

Quite tight estimates for the condition numbers $\kappa_{m,n}$
can be found in 
 \cite{D88}, \cite{E88},  \cite[Theorem 4.5]{CD05},
and \cite{ES05}. 

%------------------------------------------------------------------------------

\section{SRFT Matrices}\label{ssrft}

%------------------------------------------------------------------------------

Next we recall the definition and some basic properties of {\em SRFT} matrices, by following 
 \cite[Section 11.1]{HMT11}. 
An SRFT is an 
$n\times \rho$ complex matrix of the form $H=\sqrt{n/\rho_+}~D~\Omega~R$
where
\begin{itemize}
\item%1
$D=\diag(d_i)_{i=0}^{n-1}$ is the $n\times n$ is a
diagonal matrix, whose diagonal entries $d_i$ are independent and
uniformly distributed on the complex unit circle  $\{z:~|z|=1\}$;
\item%2
$\Omega$ is the $n\times n$ unitary 
matrix of discrete Fourier transform, $\Omega=\frac{1}{\sqrt n}(\omega^{ij})_{i,j=0}^{n-1}$
for a primitive root of unity $\omega=\exp(2\pi\sqrt{-1}/n)$; and
\item%3 
$R^T$ is a random $\rho_+\times n$ matrix that restricts an $n$-dimensional vector to 
$\rho_+$ coordinates, chosen uniformly at random, for $\rho_+\ge \rho$.
\end{itemize}
Up to scaling, an SRFT is just a section of a unitary matrix; it satisfies the norm
identity $||H||=\sqrt{n/\rho_+}$. The critical fact is that an 
appropriately designed SRFT approximately preserves the geometry of
{\em  an entire subspace of vectors.}

\begin{theorem}\label{thsrft}
{\rm The SRFT multiplier is  likely to preserve the rank
and the condition number.} Fix a $\rho\times n$ orthogonal matrix $U$
and generate an $n\times \rho_+$ SRFT matrix $H$,  where the parameter
$\rho_+=\rho_+(\rho,n)\ge \rho$  satisfies

$$4\Big (\sqrt {\rho}+\sqrt {8\log (\rho n)}\Big )^2\log(\rho)\le \rho_+\le n.$$
Then

$$0.40\le \sigma_{\rho }(UH)~~{\rm and}~~ \sigma_{1 }(UH)\le1.48$$
with the failure probability at most $O(1/\rho)$.
\end{theorem}

In words, the null space of an $n\times \rho_+$ SRFT matrix
with $\rho_+$ of order $(\rho+\log(n)\log (\rho)$ is unlikely to intersect
a fixed $\rho$-dimensional subspace. 

\begin{remark}\label{resrft}
The logarithmic
factor $\log (\rho)$ in the lower bound on $\rho_+$ can
be decreased for larger $n$  (see below), but
in contrast with the Gaussian case,
cannot 
generally be removed, that is, 
with SRFT matrices we involve 
a positive {\em oversampling integer parameter} $\rho_+-\rho$. 
For large problems, one can  reduce the numerical
constants of Theorem \ref{thsrft}. 
If $\rho\gg \log(n)$ and $\delta$ is a small positive number, then sampling
$\rho_+\ge (1+\delta) \rho \log(\rho))$
coordinates is sufficient in order to ensure that $\sigma_{\rho}(UH)\ge \delta $ with failure probability $O(\rho^{-\delta c})$ for a positive constant $c$.
 Moreover, according to \cite[Section 11.2]{HMT11}, 
the choice of $\rho_+=\rho+20$ 
is adequate in almost all applications.
\end{remark}

\begin{remark}\label{refprob}
In the case of using SRFT multipliers, 
Theorem
\ref{thsrft} bounds the failure probability by
$O(1/\rho)$.  For comparison, in the case of using Gaussian multipliers,
 the upper bound  on  the failure probability
has order $1/2^{n-\rho}$ 
by virtue of 
 Theorem \ref{thsiguna}.
\end{remark}

%------------------------------------------------------------------------------

\section{Circulant, subcirculant, and Toeplitz matrices}\label{scsct}

%------------------------------------------------------------------------------

An $n\times n$ {\em circulant} matrix
$Z=(z_{i-j\mod n})_{i,j=0}^{n-1}=\Omega^{-1}D\Omega$ 
is defined by its first column ${\bf z}=(z_{i})_{i=0}^{n-1}$ or by the diagonal matrix
 $D=\diag(d_i)_{i=0}^{n-1}$ where $(d_i)_{i=0}^{n-1}=\sqrt {n}~\Omega~ {\bf z}$
and $\Omega^{-1}=\Omega^H$ is the Hermitian transpose of $\Omega$.
The following fact links circulant and SRFT matrices.
\begin{fact}\label{facsrft}
$\sqrt{n/\rho_+}\Omega ZR$ is a SRFT matrix for 
$Z=\Omega^{-1}D\Omega$ provided that
the
 diagonal entries $d_0,\dots,d_{n-1}$ of the matrix $D$ are independent and
uniformly distributed on the complex unit circle $\{x:~|x|=1\}$
and $R$ is the random $n\times \rho$ matrix defined in the beginning of the
previous section.
\end{fact}

A circulant matrix $Z=Z({\bf z})$
is real if and only if its first column ${\bf z}$ is real.

$k\times l$ {\em Toeplitz} matrices $T=(t_{i,j})_{i,j=0}^{m-1,n-1}$ extend 
the class of circulant matrices and 
can be defined as block submatrices of $(k+l)\times (k+l)$ circulant matrices.
Such a matrix is defined by the $k+l-1$ entries of its first row and its first column.

An $n\times n$ random circulant  matrix $Z=Z({\bf z})$ tends to be 
 well-conditioned \cite{PSZ15},
and hence so do its $n\times k$ and $k\times n$ Toeplitz blocks $B$
(we call them {\em subcirculant}), defined
by the $n$ entries of their first row or column. 
Indeed, $\kappa(B)\le \kappa (Z({\bf z}))$ for such blocks $B$.

The known upper bounds on the condition number of
 a  random $n\times k$ Toeplitz matrix, defined by $n+k-1$
random entries
 of the first row and the first column, are much greater 
 (cf. \cite{PSZ15}).

We only need $O(n\log (n))$ flops in order to multiply  by a vector the $n\times n$ matrix $\Omega$,
and therefore $n\times n$ SRFT, circulant, subcirculant, and Toeplitz matrices as well.
Similar properties hold for  
$f$-circulant matrices for a complex scalar $f$ such that $|f|=1$
(cf. \cite[Section 2.6]{P01}), 
which turn into circulant matrices for $f=1$.
Using such matrices (for a fixed or random value $f$),
instead of circulant ones,
allows further variations of our algorithms.

%------------------------------------------------------------------------------

\section{Matrices Having Small Rank or Small Numerical Rank}\label{savrnk}

%------------------------------------------------------------------------------

%The following simple result (cited in Remark \ref{reclas})
%shows that the $m\times n$  matrices of a rank $\rho$
%form an algebraic variety of the dimension $d_{\rho}=(m+n-\rho)\rho$

%------------------------------------------------------------------------------

\begin{fact}\label{far1} (Cf. \cite[Proposition 1]{BV88}.)
The set $\mathbb A$ of $m\times n$ matrices of rank  $\rho$
is an algebraic variety of dimension  $(m+n-\rho)\rho$
in the space $\mathbb  R^{m\times n}$. (Clearly, $(m+n-\rho)\rho<mn$
for $\rho<\min\{m,n\}$.)
\end{fact}
\begin{proof}
Let $A$ be an $m\times n$ matrix of a rank $\rho$
with a nonsingular leading $\rho\times \rho$ block $B$
and write $A=\begin{pmatrix}
B   &   C  \\
D  &   E
\end{pmatrix}$.
Then the $(m-\rho)\times (n-\rho)$ {\em Schur complement} $E-DB^{-1}C$
must vanish, which imposes $(m-\rho)(n-\rho)$ algebraic 
equations on the entries of the matrix $A$. 
Similar argument can be applied in the case where any $\rho\times \rho$
submatrix of the matrix $A$ 
(among $\begin{pmatrix}
m      \\
\rho
\end{pmatrix}\begin{pmatrix}
n     \\
\rho
\end{pmatrix}$ such submatrices)
is nonsingular. Therefore 
$\dim \mathbb A=mn-(m-\rho)(n-\rho)=(m+n-\rho)\rho$.
\end{proof}  

%------------------------------------------------------------------------------
 
\begin{remark}\label{reclas}
How large is the class of $m\times n$ matrices 
having numerical rank $\rho$?
We characterize it 
indirectly, by noting that 
by virtue of Fact \ref{far1}
the nearby matrices
of rank $\rho$
form a variety 
of dimension $(m+n-\rho)\rho$, which increases as $\rho$
increases.
\end{remark}

{\bf Acknowledgements:}
Our work has been supported by NSF Grant CCF--1116736 and
PSC CUNY Awards 4512--0042 and 65792--0043. We are also grateful to
 a reviewer for valuable comments.
%$~$
% I wish to thank the reviewers for thoughtful and helpful comments.  
%$~$

%------------------------------------------------------------------------------

%------------------------------------------------------------------------------

\bigskip
%\bigskip

%------------------------------------------------------------------------------

%\clearpage

% - - - - - - - - - - - - - - - - - - - - - - - - - - - - - - - - - - - - -----

\end{document}